\documentclass[a4paper, 11pt]{amsart}   
\usepackage{mathptmx, amssymb,amscd,latexsym, eulervm}
\usepackage{changepage} 
\usepackage{amsmath}
\usepackage{amsthm}
\usepackage{mathdots}
\usepackage[pagebackref,colorlinks=true,linkcolor=blue,urlcolor=blue]{hyperref}
\usepackage{color}
\usepackage{comment}
\usepackage[onehalfspacing]{setspace}
\usepackage{tabularx}
\usepackage{amsfonts}
\usepackage{paralist}
\usepackage{aliascnt}
\usepackage[initials, lite]{amsrefs}
\usepackage{amscd}
\usepackage{blkarray}
\usepackage{mathbbol}
\usepackage{setspace}
\usepackage[inner=2.4cm,outer=2.4cm, bottom=3.2cm]{geometry}
\usepackage{tikz, tikz-cd}
\usepackage{calligra,mathrsfs}

\usepackage{tikz}
\usetikzlibrary{matrix}
\usetikzlibrary{arrows,calc}
\allowdisplaybreaks

\BibSpec{collection.article}{%
	+{}  {\PrintAuthors}                {author}
	+{,} { \textit}                     {title}
	+{.} { }                            {part}
	+{:} { \textit}                     {subtitle}
	+{,} { \PrintContributions}         {contribution}
	+{,} { \PrintConference}            {conference}
	+{}  {\PrintBook}                   {book}
	+{,} { }                            {booktitle}
	+{,} { }                            {series}
	+{, vol.} { }                            {volume}
	+{,} { }                            {publisher}
	+{,} { \PrintDateB}                 {date}
	+{,} { pp.~}                        {pages}
	+{,} { }                            {status}
	+{,} { \PrintDOI}                   {doi}
	+{,} { available at \eprint}        {eprint}
	+{}  { \parenthesize}               {language}
	+{}  { \PrintTranslation}           {translation}
	+{;} { \PrintReprint}               {reprint}
	+{.} { }                            {note}
	+{.} {}                             {transition}
	+{}  {\SentenceSpace \PrintReviews} {review}
}
\AtBeginDocument{%
	\def\MR#1{}
}

\makeatletter
\@namedef{subjclassname@2020}{%
	\textup{2020} Mathematics Subject Classification}
\makeatother

\newcommand{\type}{\textrm{type}}

\newcommand{\tr}{\operatorname{tr}}

\def\opn#1#2{\def#1{\operatorname{#2}}}
\opn\Cl{Cl} \opn\deg{deg} \opn\Stab{Stab} \opn\aff{aff} \opn\div{div}
\opn\cone{cone} \opn\End{End} \opn\mod{mod}  \opn\pdim{pdim} \opn\diag{diag} \opn\vert{vert} \opn\m{m} \opn\V{V} \opn\soc{Soc}
\opn\Cone{Cone} \opn\Pyr{Pyr} \opn\max{max} \opn\min{min} \opn\int{int} \opn\rev{rev} \opn\ker{ker} \opn\lat{lat} \opn\pull{pull}
\opn\cok{coker} \opn\ant{ant}
\opn\inte{int} \opn\embdim{embdim}

\newcommand{\kk}{\mathbb{k}}

\newcommand{\NN}{\normalfont\mathbb{N}}
\newcommand{\ZZ}{\normalfont\mathbb{Z}}

\newcommand{\MM}{{\normalfont\mathfrak{m}}}

\newcommand{\mm}{{\normalfont\mathfrak{m}}}
\newcommand{\fkn}{{\normalfont\mathfrak{n}}}

\newcommand{\QQ}{\mathbb{Q}}
\newcommand{\pp}{{\normalfont\mathfrak{p}}}

\newcommand{\HS}{{\normalfont\text{HS}}}

\newcommand{\depth}{\normalfont\text{depth}}

\newcommand{\codim}{{\normalfont\text{codim}}}

\newcommand{\pd}{\normalfont\text{codim}}

\newcommand{\Hom}{\normalfont\text{Hom}}

\newcommand{\indeg}{\normalfont\text{indeg}}

\newcommand{\Hilb}{{\normalfont\text{Hilb}}}

\def\f0{\mathbf{0}}

\def\1{\mathbf{1}}

\newtheorem{theorem}{Theorem}[section]

\newaliascnt{headcor}{headthm}

\aliascntresetthe{headcor}

\newaliascnt{headconj}{headthm}

\aliascntresetthe{headconj}

\newaliascnt{corollary}{theorem}
\newtheorem{corollary}[corollary]{Corollary}
\aliascntresetthe{corollary}

\newaliascnt{claim}{theorem}

\aliascntresetthe{claim}

\newaliascnt{lemma}{theorem}
\newtheorem{lemma}[lemma]{Lemma}
\aliascntresetthe{lemma}

\newaliascnt{conjecture}{theorem}

\aliascntresetthe{conjecture}

\newaliascnt{proposition}{theorem}
\newtheorem{proposition}[proposition]{Proposition}
\aliascntresetthe{proposition}

\theoremstyle{definition}
\newaliascnt{definition}{theorem}
\newtheorem{definition}[definition]{Definition}
\aliascntresetthe{definition}

\newaliascnt{notation}{theorem}

\aliascntresetthe{notation}

\newaliascnt{condition}{theorem}

\aliascntresetthe{condition}

\newaliascnt{example}{theorem}
\newtheorem{example}[example]{Example}
\aliascntresetthe{example}

\newaliascnt{examples}{theorem}

\aliascntresetthe{examples}

\newaliascnt{remark}{theorem}
\newtheorem{remark}[remark]{Remark}
\aliascntresetthe{remark}

\newaliascnt{question}{theorem}
\newtheorem{question}[question]{Question}
\aliascntresetthe{question}

\newaliascnt{questions}{theorem}

\aliascntresetthe{questions}

\newaliascnt{problem}{theorem}

\aliascntresetthe{problem}

\newaliascnt{construction}{theorem}

\aliascntresetthe{construction}

\newaliascnt{setup}{theorem}
\newtheorem{setup}[setup]{Setup}
\aliascntresetthe{setup}

\newaliascnt{algorithm}{theorem}

\aliascntresetthe{algorithm}

\newaliascnt{observation}{theorem}

\aliascntresetthe{observation}

\newaliascnt{defprop}{theorem}

\aliascntresetthe{defprop}

\newaliascnt{fact}{theorem}

\aliascntresetthe{fact}

\def\equationautorefname~#1\null{(#1)\null}
\def\sectionautorefname~#1\null{Section #1\null}
\def\subsectionautorefname~#1\null{\S #1\null}

\def\depth{\operatorname{depth}}

\def\Hom{\operatorname{Hom}}

\def\m{\mathfrak{m}}

\def\PF{\operatorname{PF}}
\def\type{\operatorname{type}}
\def\frob{\operatorname{F}}
\def\Ap{\operatorname{Ap}}
\def\tr{\operatorname{tr}}

\newcommand{\emb}{\operatorname{embdim}}

\title{Higher-dimensional Teter rings via the canonical trace}

\author{Sora Miyashita}
\address[Miyashita]
{Department of Pure And Applied Mathematics, Graduate School Of Information Science And Technology, The University of Osaka, Suita, Osaka 565-0871, Japan}
\email{u804642k@ecs.osaka-u.ac.jp}

\author{Taiga Ozaki}
\address[Ozaki]
{Department of Mathematics, School of Science, Institute of Science Tokyo,2-12-1 Ookayama, Meguro-ku, Tokyo 152-8550, Japan}
\email{taigaozaki0422.math@gmail.com}

\date{\today}

\begin{document}

	\maketitle

\begin{abstract}
We study Puthenpurakal's higher-dimensional Teter rings via the canonical trace ideal. We give a sufficient criterion for Teterness and show that, in the standard graded case, it is also necessary, yielding a characterization. Consequently, several nearly Gorenstein families are Teter; moreover, under certain hypotheses, the Cohen--Macaulay type of nearly Gorenstein rings is bounded by the codimension.
\end{abstract}

\section{Introduction}




Teter rings~(\cite{Teter1974SocleFactor}) form a classical class within Artinian Cohen–Macaulay rings and are tightly linked to the Gorenstein property.
Beyond the Artinian case, a higher–dimensional local analogue was formulated in \cite{puthenpurakal2025higher}.
This paper introduces a \emph{graded} analogue that unifies the Artinian and local viewpoints.
Throughout, let $R=\bigoplus_{i\ge 0} R_i$ be a non-Gorenstein Cohen--Macaulay graded ring with $R_0$ a field. 
Let $\omega_R$ denote the canonical module, and let $\mathfrak m:=\bigoplus_{i>0} R_i$.
In this paper, we adopt the following as our definition.
\begin{definition}\label{def:maybenice}
We say $R$ is \emph{Teter} if there exists a graded $R$–homomorphism
$\varphi:\ \omega_R \rightarrow R$
such that either $\varphi(\omega_R)=\mm$~(this happens only when $\dim(R)=0$), or $\varphi$ is injective and
$\embdim\!\bigl(R/\varphi(\omega_R)\bigr) \le \dim(R)$.
\end{definition}
It recovers the classical Artinian case at $\dim(R)=0$ and, in positive dimension, aligns with \cite{puthenpurakal2025higher}: Teter in our sense implies $R_{\mm}$ is Teter in theirs. If $R$ is a domain, the notions are equivalent, i.e., $R$ is Teter if and only if $R_{\mm}$ is Teter.
Thus \autoref{def:maybenice} furnishes a natural graded analogue.

The
{\it canonical trace}
$\tr_R(\omega_R)$ is the ideal generated by the images of all maps $\omega_R$ to
$R$; it detects the non--Gorenstein locus and has seen renewed interest.
We establish sufficient conditions for $R$ to be Teter, phrased in terms of the canonical trace;
moreover, in the standard graded case these conditions characterize Teter rings.
We fix minimal notation:
Set $r_{\min}(R):=\dim_{R_0}([\omega_R]_{-a_R})$.
$R$ is called \emph{level}~(\cite{stanley2007combinatorics}) if
$r(R)=r_{\min}(R)$, where $r(R)$ is the {\it (Cohen--Macaulay) type} and $a_R$ the {\it $a$-invariant} of $R$.
Let $\indeg(\mm):=\min \{i  \in \ZZ : [\m_R]_i \neq 0\}$ and let $\pd(R):=\embdim(R)-\dim(R)$.
Our main result is stated as follows.

\begin{theorem}[{\autoref{thm:stdgradedTeter}}]\label{thm:stdgradedTeterX}
Consider the following conditions:
\begin{enumerate}[\rm (1)]
\item $R$ is level, $r(R)\ge \pd(R)$ and $[\tr(\omega_R)]_{\indeg(\mm)}$ contains a non-zero divisor;
\item $[\omega_R]_{-a_R}$ contains a torsion-free element $x$ of $\omega_R$~(that is, $rx = 0$ implies $r = 0$ for all $r \in R$),
$r_{\min}(R) \ge \pd(R)$ and $[\tr(\omega_R)]_{\indeg(\mm)}$ contains a non-zerodivisor;
\item $R$ is a Teter and level ring with $r(R)=\pd(R)$ and $\dim(R)>0$;
\item $R$ is Teter with $\dim(R)>0$.
\end{enumerate}
Then \((1) \Rightarrow (2)\Rightarrow(3)\Rightarrow(4)\) holds.
Moreover, if \(R\) is standard graded, all the conditions are equivalent;
in particular,
if any one of them holds, then \(r(R)=r_{\min}(R)=\pd(R)\).
\end{theorem}

In this section, we give several applications of \autoref{thm:stdgradedTeterX}.
In \cite{herzog2019trace}, Herzog et al.\ introduced the notion of a \emph{nearly Gorenstein} ring by studying the trace ideal of the canonical module.
While Gorenstein rings are precisely the Cohen--Macaulay rings of type $1$, the type of a nearly Gorenstein ring is often bounded above by the codimension of $R$.

For instance, if $\codim(R)=2$ and $\dim(R)>0$, then \cite[Theorem 2.1]{ficarra2025canonical} implies that $r(R)\le 2$.
When $\codim(R)=3$, it is known that $r(R)\le 3$ for nearly Gorenstein simplicial affine semigroup rings (\cite{jafari2024nearly}).
Moreover, for nearly Gorenstein projective monomial curves one has $r(R)=3$ (\cite{miyashita2023nearly}).
In addition, nearly Gorenstein monomial curves in codimension $3$ satisfy $r(R)\le 3$ (\cite{moscariello2021nearly}).

In \cite[Question~3.7]{stamate2016betti}, Stamate asked whether the type of a nearly Gorenstein numerical semigroup ring admits an upper bound depending only on the embedding dimension (equivalently, on the codimension). Motivated by the above results, we pose the following natural question in a broader setting:

\begin{question}\label{q:intere?}
Let $R$ be a nearly Gorenstein ring. Is its type $r(R)$ bounded above by $\codim(R)$?
In particular, if $R$ is nearly Gorenstein and $\codim(R)=3$, does it follow that $r(R)\le 3$?
\end{question}

By applying \autoref{thm:stdgradedTeterX} together with several known results, we answer \autoref{q:intere?} affirmatively under suitable hypotheses.
Recall the notation $r_{\min}(R):=\dim_{R_0}([\omega_R]_{-a_R})$.


\begin{corollary}[{\autoref{cor:omgCMtype}}]\label{cor:omgCMtypeX}
Assume $R$ is (non-Gorenstein) nearly Gorenstein and that
$R_{\indeg(\mm_R)}$ contains a non-zero divisor of $R$~(e.g., $R$ is a domain,
or a standard graded ring with $\dim(R)>0$).
Then the following hold:
\begin{itemize}
\item[\rm (1)]
If $R$ is level,
then $r(R) \le \pd(R)$;
\item[\rm (2)]
If $r_{\min}(R) \ge \pd(R)$,
then $R$ is level and $r(R)=\pd(R)$;
\item[\rm (3)]
If $R$ is standard graded, $\pd(R)=3$, $\dim(R) \ge 2$ and $r_{\min}(R) \neq 2$, then $R$ is level and $r(R)=3$;
\end{itemize}
\end{corollary}

In \cite[Section 5]{puthenpurakal2025higher}, it is shown that several classes of rings, including rings of finite Cohen--Macaulay type and one-dimensional rings of minimal multiplicity, are Teter in the higher-dimensional sense.
As an application of \autoref{thm:stdgradedTeterX}, we show that various standard graded nearly Gorenstein rings are Teter.
Furthermore, we show that the Teter property can be constructed and is preserved under operations such as fiber products and certain Veronese subalgebras.




\begin{corollary}[{See a part of \autoref{thm:examples} and \autoref{thm:Fiber2} and \autoref{thm:Veronese2}}]\label{thm:examplesX}
The following hold:
\begin{enumerate}
\item[\rm (a)]
Assume that $R$ is a standard graded nearly Gorenstein.
Then $R$ is Teter, if it satisfies one of the following conditions:
\begin{enumerate}[\rm (1)]
\item \(R\) is a level semi-standard graded ring with \(r(R)=\pd(R)\) and $\dim(R)>0$;
\item \(R\) is an affine semigroup ring with $\pd(R) \le 3$ and $\dim(R)=2$;
\item \(R\) has minimal multiplicity;
\item \(R\) is a Stanley--Reisner ring;
\end{enumerate}
\item[\rm (b)]
Let $A$ and $B$ be 1-dimensional Cohen--Macaulay generically Gorenstein standard graded rings, at least one of which is not regular. 
Then the fiber product $A \times_\kk B$ (see \autoref{def:fiberrrrrrrrrr})is Teter if and only if both $A$ and $B$ have minimal multiplicity.
\item[\rm (c)] Assume $R$ is a (non-Gorenstein) generically Gorenstein standard graded ring with minimal multiplicity.
       If either $\dim(R)=1$ or $\tr(\omega_R)=\mm$ with $\dim(R)=2$, then $R^{(k)}$ is Teter for every $k\ge 2$.
\end{enumerate}
\end{corollary}

While standard graded Teter rings satisfy $r(R) = r_{\min}(R)$ by \autoref{thm:stdgradedTeterX}, this equality fails for semi-standard graded rings. Specifically, for any $a \in \mathbb{Z}_{>0}$, there exists a nearly Gorenstein Teter ring $R$ such that $r(R) - r_{\min}(R) = a$ (\autoref{thm:examples}).

Finally, we investigate the Teter property of numerical semigroup rings, which are a typical class of rings that are not even semi-standard graded. As a main result of this section, we show that the Teter property of numerical semigroup rings can be characterized by using the minimal set of generators and invariants determined by the numerical semigroup called pseudo-Frobenius numbers (see \autoref{thm:numeOZAKI}).
Using this characterization, we obtain the following assertions for the cases where the numerical semigroup has minimal multiplicity or the codimension is $2$
(see \autoref{cor:minimalnumefunny} and \autoref{propprop:numenume}).
Furthermore, we obtain the following result analogous to \autoref{cor:mmOMG} (see \autoref{prop:cor:minimalnumefunny1}).


\subsection*{Outline}
This paper is organized as follows. 
\autoref{section2} reviews the basics of Cohen--Macaulay graded rings, trace ideals, and numerical semigroup rings. 
In \autoref{sect3}, we present our main result on the Teter property for (semi-)standard graded rings. 
\autoref{sect4} explores the relationship between nearly Gorenstein and Teter rings, establishing sufficient conditions. 
\autoref{sect5} examines how the Teter property behaves under ring operations, specifically fiber products and Veronese subalgebras. 
Finally, \autoref{sect6} provides a characterization of the Teter property for numerical semigroup rings.

\begin{setup}\label{setup1}
Throughout this paper,
we denote the set of non-negative integers by $\NN$.
Let $R=\bigoplus_{i \in \NN} R_i$ be a positively graded ring. Unless otherwise stated, we assume that $R_0$ is a field. Hence, $R$ has the unique graded maximal ideal given by $\MM_R:=\bigoplus_{i \in \NN \setminus \{0\}} R_i$.
Let $M$ and $N$ be graded $R$-modules
and let $k \in \ZZ$.
{Let $M(k)$ denote the graded $R$-module having the same underlying $R$-module structure as $M$,
where $[M(k)]_n = [M]_{n+k}$ for all $n \in \ZZ$.}
Moreover, let ${}^*\Hom_R(M, N)$ denote the graded $R$-module consisting of graded homomorphisms from $M$ to $N$.
Moreover, we set $M^{\vee}:={}^*\Hom_R(M,R)$.
Let \( \mathcal{S} \subset R \) be the multiplicative set of homogeneous non-zero divisors. 
The \emph{graded total quotient ring} \( {}^*Q(R) \) is defined as the localization
${}^*Q(R) := \mathcal{S}^{-1}R$.
Then \( Q(R) \) inherits a natural \( \mathbb{Z} \)-grading, with graded components given by
$[Q(R)]_i := \left\{ \frac{r}{s} \in Q(R) : r, s \in R\text{\;are homogeneous},\ \deg(r) - \deg(s) = i \right\}$ for $i \in \mathbb{Z}$.
Suppose that $R$ is Cohen--Macaulay. 
Let $\omega_R=\bigoplus_{i \in \ZZ} [\omega_R]_i$ denote the graded canonical module of $R$ (see \autoref{defcanon}). Let $a_R$ denote the {\it $a$-invariant} of $R$, that is, 
\[
a_R := - \min \{ i \in \mathbb{Z} : [\omega_R]_i \neq 0 \}.
\]
\end{setup}

\section{Preliminaries}\label{section2}
The purpose of this section is to lay the groundwork for the proof of our main result.
Throughout, we work under \autoref{setup1} in all subsections.




\subsection{Canonical modules over Cohen--Macaulay graded rings}
Let us recall the definition of the canonical module over a Cohen--Macaulay positively graded ring.
Throughout this subsection, we retain \autoref{setup1}, and we further assume that \( R \) is Cohen--Macaulay.
Let us recall the definition of the canonical module.

\begin{definition}{\cite[Definition (2.1.2)]{goto1978graded}}\label{defcanon}
The finitely generated graded $R$-module
\[
\omega_R:={}^*\Hom_R(\mathrm{H}_{\MM_R}^{\dim(R)}(R), E_R)
\]
is called the {\it canonical module}, where $\mathrm{H}_{\MM_R}^{\dim(R)}(R)$ denotes the $d$-th local cohomology and $E_R$ denotes the injective envelope of $\kk$.
\end{definition}



\begin{definition}
In this paper, for simplicity of notation, we denote the {\it (Cohen--Macaulay) type} of $R$ by $r(R)$, and set
$r_{\min}(R):=\dim_{R_0}([\omega_R]_{-a_R})$.
Note that $r_{\min}(R) \le r(R)$ always holds.
$R$ is called {\it level}~(\cite{stanley2007combinatorics}) if
$\omega_R$ is generated in a single degree;
equivalently, $r(R)=r_{\min}(R)$ (see, e.g., \cite
[Proposition 3.2]{stanley2007combinatorics}).
Moreover,
following the terminology of \cite{ene2015pseudo},
we say that $R$ is {\it pseudo-Gorenstein} if
$r_{\min}(R)=1$.
\end{definition}

\subsection{Semi-standard graded rings, Hilbert series, and almost Gorenstein graded rings}

Throughout this subsection we keep \autoref{setup1} in force, and in addition assume that $R$ is Noetherian.

\begin{definition}
We say that $R$ is {\it standard graded} if $R=R_0[R_1]$, i.e., if $R$ is generated by $R_1$ as an $R_0$-algebra.
We say that $R$ is {\it semi-standard graded} if $R$ is finitely generated as an $R_0[R_1]$-module.
\end{definition}

\begin{remark}
The notion of semi-standard graded rings, which extends standard graded rings, appears naturally here from the viewpoint of combinatorial commutative algebra.
Typical examples include the \emph{Ehrhart rings} of lattice polytopes and the \emph{face rings} of simplicial posets (see \cite{stanley1991f}).
\end{remark}

Recall that when $R$ is semi-standard graded, its Hilbert series
\[
\Hilb(R,t):=\sum_{i=0}^\infty \dim_{R_0}(R_i)\,t^i
\]
has the form
\[
\Hilb(R,t)=\frac{h_0+h_1 t+\cdots+h_s t^s}{(1-t)^{\dim(R)}},
\]
where $h_i\in\mathbb{Z}$ for all $0\le i\le s$, $h_s\ne 0$, and $\sum_{i=0}^s h_i\ne 0$.
The integer $s$, called the {\it socle degree} of $R$, is denoted by $s(R)$.
The sequence $h(R)=(h_0,h_1,\ldots,h_{s(R)})$ is the {\it $h$-vector} of $R$.

\begin{remark}
Assume that $R$ is semi-standard graded, and let $h(R)=(h_0,h_1,\ldots,h_{s(R)})$ be its $h$-vector.
We always have $h_0=1$ since $h_0=\Hilb(R,0)=\dim_{R_0}(R_0)=1$.
Moreover, if $R$ is Cohen--Macaulay, its $h$-vector satisfies $h_{s(R)}=r_{\min}(R)$ and $h_i\ge 0$ for every $0\le i\le s(R)$.
For further background on $h$-vectors of Noetherian graded rings, see \cite{stanley1978hilbert,stanley2007combinatorics}.
\end{remark}

\begin{definition}
Assume that $R$ is Cohen--Macaulay.
$R$ is called {\it almost Gorenstein}~(\cite[Section 10]{goto2015almost}) if there exists a degree-$0$ graded $R$-monomorphism
$\varphi\colon R\hookrightarrow \omega_R(-a_R)$ such that $C:=\cok(\varphi)$ is either the zero module or an Ulrich $R$-module,
i.e., $C$ is Cohen--Macaulay and $\mu(C)=e(C)$.
Here, $\mu(C)$ (resp.\ $e(C)$) denotes the minimal number of generators of $C$ (resp.\ the multiplicity of $C$ with respect to $\MM_R$).
\end{definition}


\subsection{Trace ideals over graded rings}\label{lastsubsec}
Throughout this subsection, we retain \autoref{setup1}, and assume that \( R \) is Noetherian.
Let $M:=\bigoplus_{i \in \ZZ} M_i$ be a (not necessarily finitely generated) $\ZZ$-graded $R$-module.

\begin{definition}\label{def:GRADEDcanon}
The sum of all images of graded homomorphisms $\varphi \in {}^*\Hom_R(M,R)$ is called the {\it trace ideal} of $M$:
\[
\tr_R(M):=\sum_{\varphi \in {}^*\Hom_R(M,R)}\varphi(M).
\] 
When there is no risk of confusion about the ring, we simply write $\tr(M)$.
\end{definition}

\begin{remark}
Let $I$ be a graded ideal of $R$ containing
a non-zero divisor of $R$.
Set
$$\displaystyle I^{-1}:=\{x \in {}^*Q(R):xI\subset R\}.$$
Then we have $\tr_R(I)=I \cdot I^{-1}$
(see the proof of \cite[Lemma 1.1]{herzog2019trace}).
\end{remark}

\begin{remark}\label{rem:gradedaoyamagoto}
Assume $R$ is Cohen--Macaulay and let $\omega_R$ be the canonical module of $R$.
It is known that \(\tr_R(\omega_R)\) describes the non-Gorenstein locus~(see \cite[Lemma 2.1]{herzog2019trace}).
In particular,
$R$ is Gorenstein if and only if $\tr_R(\omega_R)=R$.
\end{remark}

\begin{definition}[{\cite[Definition 2.2]{herzog2019trace}}]
Assume that $R$ is Cohen--Macaulay.
$R$ is called \textit{nearly Gorenstein} if
$\tr_R(\omega_R) \supset {\MM_R}$.
\end{definition}

At the end of this section, we recall the definition of the Veronese subalgebra and present a lemma concerning its relation to the trace ideal of the canonical module.

\begin{definition}
Let \( k > 0 \) be a positive integer. For \( R = \bigoplus_{i \in \NN} R_i \), the \textit{\( k \)-th Veronese subalgebra} of $R$ is defined by
$R^{(k)} := \bigoplus_{i \in \NN} R_{ik}$.
Notice that the grading on \( R^{(k)} \) is given by \( [R^{(k)}]_i = R_{ik} \) for each \( i \in \NN \).  
Similarly, we define its \textit{\( k \)-th Veronese submodule} of $M$ by
$M^{(k)} := \bigoplus_{i \in \ZZ} M_{ik}$.
\end{definition}

\begin{definition}
Let ${}^*Q(R)$ denote the graded total quotient ring
(see \autoref{setup1}).
{\it A graded fractional ideal} of 
$R$ is a finitely generated graded
$R$-submodule $J \subset {}^*Q(R)$
that contains a homogeneous non-zero divisor of $R$.
For a graded fractional ideal $J$ of $R$, we set
$J^{-1}=\{a \in {}^*Q(R) : aJ \subset R\}$.
\end{definition}

\begin{lemma}[{\cite[Theorem 3.4]{miyashita2024linear}}]\label{lem:Yeah}
Let $\fkn = (R_1)R$. Let $I$ be a graded ideal generated by a subset of $R_1$, and let $J$ be a graded fractional ideal of $R$. Suppose that there exists a homogeneous non-zero divisor $y \in R_1$.
If
$\tr_{R}(J) \supset I$,
then we have $\tr_{R ^{(k)}}(J ^{(k)}) \supset (\fkn^{k-1}I) ^{(k)}$ for any $k>0$.
\end{lemma}

The next result is a generalization of \cite[Proposition~3.7]{miyashita2025canonical}.
Its proof is essentially the same as that of \cite[Proposition~3.7]{miyashita2025canonical}, but we include it for the reader's convenience.

\begin{proposition}\label{p:Wednesday}
Let $R$ be a Noetherian graded ring and let $J \subset {}^*Q(R)$ be a fractional ideal of $R$.
Suppose that
$J_{\indeg_R(J)}$ contains an $R$-torsion free element of $J$.
If $[\tr(J)]_{\indeg_R(\mm)}$ contains a non-zero divisor, then
there exists a graded ideal $I \subset R$ such that
$J \cong I$ and
$\dim_{R_0}(I_{\indeg_R(\mm)})= \dim_{R_0} \left(J_{\indeg_R(J)} \right)$.
\end{proposition}
\begin{proof}
We may assume that $R_0$ is infinite.
Assume that
there exists a non-zero divisor $z \in [\tr(J)]_{\indeg_R(\mm)}$.
Put $n=\dim_\kk(J_{\indeg_R(J)})$ and $a=\indeg_R(J)$.
Then by \cite[Lemma 2.12~(2)]{miyashita2024linear},
we may assume that $J$ is minimally generated by $R$-regular elements $f_1,\cdots,f_m$, where $m \ge n$.
	We may assume that $f_1,\cdots,f_n \in J_a$.
    Note that $\tr(J)=J \cdot J^{-1}$ by
    \cite[Lemma 2.6]{miyashita2024linear}.
    Since $z \in [\tr(J)]_{\indeg_R(\mm)}$, we can write $z=\sum_{i=1}^m{f_i g_i}$, where $g_1,\cdots,g_m$ are homogeneous elements of $J^{-1}$.
    Moreover, we have $g_i=0$ for any $n < i \le m$ because $f_1,\cdots,f_n \in J_a$ are $R$-regular. Therefore the vector space $V=[J^{-1}]_{\indeg_R(\mm)-a}$ is equal to $\frac{1}{b} \langle x_1,\cdots,x_s \rangle$ where $b$ is a non-zero divisor of $R$, $s$ is natural number and $x_1,\cdots,x_s$ homogeneous elements of $R$ of the same degree, $\indeg_R(\mm)-a+\deg(b)$.
	
Then the vector space $V'=\langle x_1,\cdots,x_s \rangle$ contains a non-zero divisor. Otherwise, $V'$ is contained in some associated prime $\pp$ of $R$,
and so $b[\tr(J)]_{\indeg_R(\mm)} R\subset J V' \subset \pp$, thus $[\tr(J)]_{\indeg_R(\mm)} R \subset \pp$. This yields a contradiction.
Therefore, there exists an $R$-torsion free element $\frac{y}{b} \in V$, so that $I=\frac{y}{b} \cdot J \cong J$ as a graded $R$-module
and
$\indeg_R(I)=\indeg_R(\mm)$.
In particular,
we have
$\dim_{R_0}(I_{\indeg_R(\mm)})=\dim_{R_0}(I_{\indeg_R(I)})=\dim_{R_0}(J_{\indeg_R(J)})$,
as desired.
\end{proof}

\subsection{Numerical semigroup rings}\label{subsec:nume}
Let $\kk$ be a field. For a natural number $n\in\mathbb{N}$, we set $[n]:=\{1,\ldots,n\}$.\\
A subsemigroup $H$ of $\mathbb{N}$ such that $\#(\mathbb{N}\setminus H)<\infty$ is called a numerical semigroup.
It is known that a numerical semigroup has a unique minimal system of generators as a semigroup.
Therefore, a given numerical semigroup $H$ can always be written in the form $H=\langle a_1,\ldots,a_n\rangle$ where $\gcd(a_1,\ldots,a_n)=1,\,a_1<\ldots<a_n$.
Here, $\emb(H)$ denotes the number of minimal generators of $H$.\\
Hereafter, any numerical semigroup is assumed to be represented in this form and $n=\embdim(H)$. For a given numerical semigroup $H$, we define $R_H:=\kk[H]=\kk[t^{a_1},\ldots,t^{a_n}]\subset \kk[t]$.\\
The ring $R_H$ admits a graded ring structure by setting $\deg t = 1$. 
Let $S = \kk[X_1, \dots, X_n]$ be the polynomial ring in $n$ variables, and consider the graded surjective homomorphism $\varphi \colon S \to R_H$ by $X_i\mapsto t^{a_i}$. 
Then, the kernel $\ker \varphi$ is called the defining ideal of $R_H$. 
In the case $n=3$, it is known by 
\cite{herzog1970generators} that the defining ideal is generated by the $2$-minors of a $2 \times 3$ matrix $A$ over $S$. 
This matrix is called the {\it Hilbert--Burch matrix}.
Furthermore, for a given numerical semigroup, the pseudo-Frobenius numbers $\PF(H)$ and the Ap\'ery set $\Ap(H,a_1)$ are defined as follows:
\[\PF(H):=\{h\in H\,:\,\alpha+h\in H\,\; {\it for\; all} \;\,h\in H\},\]
\[\Ap(H,a_1)=\{h\in H : h-a_1\in H\}.\]
Furthermore, we define $\frob(H) := \max (\PF(H))$, which is called the Frobenius number of $H$.\\
This invariant $\PF(H)$ corresponds to the system of generators of the canonical module $\omega_{R_H}$ of $R_H$ as an $R_H$-module. That is,
\[\omega_{R_H}=\sum_{h\in \mathbb{N}\setminus H}{R_H t^{-h}}=\sum_{\alpha\in \PF(H)}{R_H t^{-\alpha}}=\sum_{\alpha\in\mathbb{N}\setminus H}\kk t^{-\alpha}+\sum_{h\in \mathbb{Z}_{>0}}{\kk t^{h}},\]
and we have $\operatorname{r}(R_H)=\#(\PF(H))=:\type(H)$.\\

Several classes of numerical semigroup rings are characterized using $\PF(H)$.
\begin{remark}
    Let $H$ be a numerical semigroup. We define the following classes:
    \begin{enumerate}[\rm(1)]
        \item We say that $H$ is \textit{symmetric} if $R_H$ is Gorenstein. This is characterized by the condition $\type(H)=1$.
        \item Suppose that $\frob(H)$ is even. We say that $H$ is \textit{pseudo-symmetric} if $\PF(H)=\left\{\frac{\frob(H)}{2}, \frob(H)\right\}$.
        \item Let $\PF(H)=\{\alpha_1, \alpha_2, \ldots, \alpha_r = \frob(H)\}$ with $\alpha_1 < \ldots < \alpha_r$. 
        We say that $H$ is \textit{almost symmetric} if
        $\frob(H) = \alpha_{r-i} + \alpha_i$ holds for any $1 \leq i \leq r-1$.
        Note that $H$ is almost symmetric if and only if $R_H$ is almost Gorenstein
        (see \cite[Proposition 2.3]{goto2013almost}).
    \end{enumerate}
\end{remark}

Finally, let $\Omega_H^+$ be the set corresponding to the generators of $\omega_{R_H}$:
\[\Omega_H^+=\{m\in \mathbb{N} : t^m\in\omega_{R_H}\},\]
and let $\Omega_H^-$ be the set corresponding to the generators of $\omega_{R_H}^{-1}$:
\[\Omega_H^-=\{m\in \mathbb{N} :  m-\alpha_i\in H, for\,all\, i\in [n]\}.\]
According to \cite[Lemma1.1]{herzog2019trace}, the canonical trace ideal $\tr_{R_H}(\omega_{R_H})$ of $R_H$ is obtained by $\tr_{R_H}(\omega_{R_H})=\omega_{R_H}\omega_{R_H}^{-1}$. Thus, the set $\tr_{R_H}(H)$ corresponding to the generators of $\tr_{R_H}(\omega_{R_H})$ over $\kk$ is given by $\tr(H)=\Omega_H^+ +\Omega_H^-$.

\section{Graded Teter rings and canonical traces}\label{sect3}
In this section, we define the graded Teter ring as follows so as to ensure consistency between puthenpurakal's definition of higher-dimensional Teter local rings~(\cite{puthenpurakal2025higher}) and the classical definition of Artinian Teter rings.
Our goal is then to clarify a relationship between Teter rings and the trace ideal of the canonical module (see \autoref{thm:stdgradedTeter}).
Throughout this section, we retain \autoref{setup1} and use the notation of \autoref{section2}.
In addition, we assume throughout that $(R,\mathfrak m)$ is a positively graded Cohen--Macaulay ring.

\begin{definition}\label{main}
Assume $R$ is not Gorenstein.
We say that $R$ is {\it Teter}
if
there exists a graded $R$-homomorphism $\varphi: \omega_R \rightarrow R$ such that $\varphi(\omega_R)=\mm$,
or
$\varphi$ is monomorphism
and
$\embdim(R/\varphi(\omega_R)) \leq \dim(R)$.
\end{definition}

\begin{definition}
Assume that $R$ is not Gorenstein.
$R$ is called of {\it Teter type}~(\cite{gasanova2022rings})
if there exists $\varphi \in {}^*\Hom_R(\omega_R,R)$
such that $\tr_R(\omega_R)=\varphi(\omega_R)$.
It is known that $R$ is not of Teter type if it is generically Gorenstein~(see \cite[Theorem 1.1]{gasanova2022rings}).
\end{definition}

\begin{remark}\label{rem:3.2}
\begin{enumerate}[\rm (a)]
\item Assume that $R$ is non-Gorenstein and $\dim(R)>0$.
Let $\varphi \in {}^*\Hom_R(\omega_R,R)$.
Then we have
$\varphi(\omega_R) \neq \mm$.
\item Assume that $\dim(R)=0$ and $R$ is not Gorenstein, then the following are equivalent:
\begin{enumerate}[\rm (1)]
  \item $R$ is Teter in the sense of \autoref{main};
  \item There is an $R$-surjection $\omega_R \rightarrow \mm$ as a graded $R$-module, that is,
  $R$ is Teter in the sense of \cite{Teter1974SocleFactor}.
\end{enumerate}
\item Assume that $\dim(R)>0$ and $R$ is not Gorenstein.
Consider the following conditions:
\begin{itemize}
\item[\rm (1)] There is a graded ideal $I \subsetneq R$ such that $\omega_{R} \cong I$ as a graded $R$-module and $\embdim(R/I) \leq \dim(R)$;
\item[\rm (2)] There is an ideal $J \subsetneq R_\m$ such that $\omega_{R_\m} \cong J$ as an $R_\mm$-module and $\embdim(R_\m/J) \leq \dim(R_\m)$;
\item[\rm (3)] There is an ideal $\widehat{J} \subsetneq \widehat{R_\m}$ such that
$\omega_{\widehat{R_\m}} \cong \widehat{J}$ as an $\widehat{R_\m}$-module and
$\embdim(\widehat{R_\m}/\widehat{J}) \leq \dim(\widehat{R_\m})$;
\item[\rm (4)] $R$ is a Teter graded ring in our sense;
\item[\rm (5)] $R_\m$ is a Teter local ring in the sense of \cite{puthenpurakal2025higher}.
\end{itemize}
Then
\((1)\Leftrightarrow(2)\Leftrightarrow (3) \Leftrightarrow (4)\Rightarrow(5)\) holds.
Moreover, if $R$ is a domain, then all the conditions are equivalent.
\end{enumerate}
\end{remark}
\begin{proof}
(a):
Assume $\varphi(\omega_R)=\mm$.
Since $\tr_R(\omega_R)=\mm$ and $\dim(R)>0$,
$R$ is generically Gorenstein by \cite[Lemma 2.1]{herzog2019trace}.
Therefore, by \cite[Theorem~1.1]{gasanova2022rings}, $R$ is not of Teter type; hence $\varphi(\omega_R)\neq\mathfrak m$, a contradiction.

(b):
(1) $\Rightarrow$ (2):
By definition, there exists a graded $R$-homomorphism $\varphi:\omega_R\to R$ such that either
\[
\varphi(\omega_R)=\mathfrak m
\quad\text{or}\quad
\varphi \text{ is injective and } \operatorname{embdim}\!\bigl(R/\varphi(\omega_R)\bigr)\le \dim(R).
\]
If $\varphi(\omega_R)=\mathfrak m$, we are done. Otherwise, we may assume that $\varphi$ is injective and
$\operatorname{embdim}\bigl(R/\varphi(\omega_R)\bigr)\le 0$. In this case one has
$\mathfrak m\subset \varphi(\omega_R)$. Since $R$ is not Gorenstein, we also have
$\varphi(\omega_R)\neq R$. Hence $\varphi(\omega_R)=\mathfrak m$.

(2) $\Rightarrow$ (1):
Let $\psi:\omega_R\to \mathfrak m$ be a surjection and let $\iota:\mathfrak m\hookrightarrow R$ be the natural inclusion. Set $\varphi=\iota\circ\psi:\omega_R\to R$. Then
$\operatorname{embdim}\!\bigl(R/\varphi(\omega_R)\bigr)=\operatorname{embdim}(R/\mathfrak m)=0$,
so (1) holds.

(c):
\((1)\Leftrightarrow(2)\Leftrightarrow (3)\) follows from \cite[Corollary~3.9]{hashimoto2023indecomposability}.

(1) $\Rightarrow$ (4):
Choose $\varphi:\omega_R\to R$ as the composite of an isomorphism $\omega_R\cong I$ with the inclusion $I\hookrightarrow R$.
Then
$\operatorname{embdim}\!\bigl(R/\varphi(\omega_R)\bigr)=\operatorname{embdim}(R/I)\le \dim(R)$,
and hence $R$ is Teter.

(4) $\Rightarrow$ (1):
By (a) and definition,
there exists
a graded $R$-homomorphism $\varphi:\omega_R\to R$ such that 
$\varphi$ is injective and $\operatorname{embdim}\bigl(R/\varphi(\omega_R)\bigr)\le \dim(R)$.
Set $I=\varphi(\omega_R)$. As $R$ is not Gorenstein, we have $I\neq R$; consequently $\omega_R\cong I$ and $\operatorname{embdim}(R/I)\le \dim(R)$, establishing (1).

Finally, \((2)\Rightarrow(5)\) follows from \cite[Theorem~2.1]{puthenpurakal2025higher}. If, in addition, $R$ is a domain, then \((5)\Rightarrow(2)\) is given by \cite[Theorem~1.1]{puthenpurakal2025higher}.
\end{proof}

\begin{remark}\label{rem:localgradedooo1}
In \autoref{rem:3.2}~(c), (4) $\Rightarrow$ (3) does not hold in general.
We will prove later using a criterion for Teterness (see \autoref{rem:localgradedooo2}).
\end{remark}

\begin{corollary}\label{cor:gradedtype}
If $R$ is Teter,
then
$r(R)=\pd(R)$.
\begin{proof}
When $\dim(R)=0$, it follows from \autoref{rem:3.2}~(a) that $R$ is Teter in the sense of \cite{Teter1974SocleFactor}. In this case, \cite[Theorem~3.4]{elias2017teter} yields
$r(R)=\embdim(R)=\pd(R)$.
Hence it suffices to treat the case $\dim(R)>0$. By \autoref{rem:3.2}~(b), the localization $R_{\mm}$ is a local Teter ring. Therefore, by \cite[Proposition~5.15]{puthenpurakal2025higher} we have
$r(R_{\mm})=\embdim(R_{\mm})-\dim(R_{\mm})$.
Since
$r(R_{\mm})=r(R)$,
$\embdim(R_{\mm})=\embdim(R)$,
$\dim(R_{\mm})=\dim(R)$
(see \cite[Proposition~1.5.15]{bruns1998cohen}),
we conclude that
$r(R)=\embdim(R)-\dim(R)=\pd(R)$.
\end{proof}
\end{corollary}

Recall that
$r_{\min}(R):=\dim_{R_0}([\omega_R]_{-a_R})$.

\begin{corollary}\label{p:Wednesday1}
Assume that
$[\omega_R]_{-a_R}$ contains an $R$-torsion free element of $\omega_R$.
If $[\tr(\omega_R)]_{\indeg_R(\mm)}$ contains a non-zero divisor, then
there is a graded ideal $I$ of $R$
such that
$\omega_R \cong I$ as a graded $R$-module
and
$r_{\min}(R) \le \embdim(R)-\embdim(R/I).$
\end{corollary}
\begin{proof}
By \autoref{p:Wednesday}, there exists a graded ideal \(I\subsetneq R\) such that \(\omega_R \cong I\). Then,
$\embdim(R/I)\le \embdim(R)-r_{\min}(R)$.
This yields the desired inequality.
\end{proof}

\begin{corollary}\label{p:Wednesday2}
Assume that
$r_{\min}(R) \ge \pd(R)$,
$[\omega_R]_{-a_R}$ contains an $R$-torsion free element of $\omega_R$, and that $[\tr(\omega_R)]_{\indeg_R(\mm)}$ contains a non-zero divisor,
then $R$ is Teter.
\begin{proof}
By \autoref{p:Wednesday1} and the standing assumption,
there exists a graded ideal $I$ of $R$ such that
$\omega_R \cong I$ as a graded $R$-module and
\[
\embdim(R)-\dim(R)=\pd(R)\le r_{\min}(R)\le\embdim(R)-\embdim(R/I).
\]
In particular, $\embdim(R/I)\le \dim(R)$, and hence, $R$ is Teter by \autoref{rem:3.2}\,(b).
\end{proof}
\end{corollary}

\begin{remark}\label{rem:obs}
Assume $R$ is standard graded and let $I \subset R$ be a graded ideal.
Note that $\embdim(R/I)=\embdim(R)-\dim_{R_0}(I_1)$ because $R$ is standard graded.
Therefore,
by using Auslander--Buchsbaum formula~(see \cite[Theorem 1.3.3]{bruns1998cohen}),
we have
$\embdim(R/I) \le \dim(R)$
if and only if $\dim_{R_0}(I_1)
\geq \pd(R)$.
\end{remark}

\begin{proposition}\label{rem:mazide?}
Assume that $R$ is standard graded
and $\dim(R)>0$.
If $R$ is Teter,
then it is level, $r(R)
\ge \pd(R)$
and
$[\tr_R(\omega_R)]_1$ contains a non-zero divisor of $R$.
\begin{proof}
Since $R$ is Teter, by \autoref{rem:obs} there exists a graded ideal $I \subset R$ such that $\omega_R \cong I$ and $\dim_{R_0}(I_1) \ge \pd(R)$.
In this case $I_1 \neq 0$, and since
$r_{\min}(R)=\dim_{R_0}(I_1)\ge \pd(R)$,
we have $r_{\min}(R)\ge \pd(R)$.
Therefore, by \autoref{cor:gradedtype} we obtain
$\pd(R) \le r_{\min}(R) \le r(R) = \pd(R)$.
Hence $r(R)=r_{\min}(R)$, so $R$ is level.
Because $R$ is Teter with $\dim(R)>0$, it is in particular generically Gorenstein by \autoref{rem:3.2}~(c).
Since $R$ has positive Krull dimension, is level, and is generically Gorenstein, the same argument as in the proof of \cite[Theorem~4.4.9]{bruns1998cohen} shows that $I_1\ (\cong [\omega_R]_{-a_R})$ contains a non-zerodivisor of $R$.
Consequently $[\tr_R(\omega_R)]_1=[\tr_R(I)]_1$ contains $I_1$ and, in particular, contains a non-zerodivisor of $R$.
\end{proof}
\end{proposition}

\begin{corollary}\label{cor:eeyan}
All Cohen--Macaulay standard graded Teter rings are level.
\end{corollary}
\begin{proof}
Let $R$ be a Cohen--Macaulay standard graded Teter ring.
By \autoref{rem:mazide?}, if $\dim(R)>0$ then $R$ is level, so it suffices to consider the case $\dim(R)=0$.
Although the fact that an Artinian graded Teter ring is level is probably well known to experts, we include a proof for completeness.
In this case there exists a standard graded Artin Gorenstein ring $G$ such that
$R \cong G/\soc(G)$ and $\embdim(R)=\embdim(G)$,
where $\soc(G):=\{a \in G: a\mm_G=0\}$.
Let $s$ be the socle degree of $G$ and put $n=\embdim(R)$.
Then the $h$-vectors of $G$ and $R$ can be written as
$h(G)=(h_0,\dots,h_{s(G)})$ and $h(R)=(h_0,\dots,h_{s(G)-1})$.
Since $G$ is Gorenstein, its $h$-vector is symmetric~(see \cite[Theorem 4.1]{stanley1978hilbert}), hence
$h_{s(G)-1}=h_1=\dim_{\kk} R_1=n$.
Therefore $R$ is level by \cite[Theorem~3.12]{elias2017teter}.
\end{proof}

\begin{remark}\label{rem:localgradedooo2}
\begin{enumerate}[\rm (a)]
\item \autoref{cor:eeyan} fails in general in the semi-standard graded setting, and there exist examples in which the difference between $r(R)$ and $r_{\min}(R)$ can take any prescribed value (see \autoref{thm:examples}).
\item In \autoref{rem:3.2}~(c), (4) $\Rightarrow$ (3) does not hold in general.
Let $A=\kk[x,y]/(x^3,y^4,xy)$. Note that $B=A_{\mm}$ is an Artinian Teter local ring in the sense of \cite{Teter1974SocleFactor} (see \cite[Example 4.3\,(2)]{gasanova2022rings}).
Set $R=A[z]$ and regard $R$ as a standard graded ring by declaring $\deg(x)=\deg(y)=\deg(z)=1$. By \cite[Example 5.15]{puthenpurakal2025higher},
we have
$R_{\mm_R} \cong B[z]_{\mm_B B[z] + z B[z]}$ is a Teter local ring in the sense of \cite{puthenpurakal2025higher}. However, since $R$ is not level, $R$ is not Teter by \autoref{cor:eeyan}.
\end{enumerate}
\end{remark}


\begin{remark}
\label{rem:kakubeki?}
If $R$ is generically Gorenstein and level,
then $[\omega_R]_{-a_R}$ contains a torsion-free element of $\omega_R$.
\begin{proof}
While the statement above differs from \cite[Theorem 4.4.9]{bruns1998cohen}, the argument in its proof yields the claim needed here. Moreover, the standard graded hypothesis is not used and can be omitted verbatim.
\end{proof}
\end{remark}

We are now in a position to prove our main theorem.

\begin{theorem}[{\autoref{thm:stdgradedTeterX}}]\label{thm:stdgradedTeter}
Assume that $R$ is not Gorenstein. Consider the following conditions:
\begin{enumerate}[\rm (1)]
\item $R$ is level, $r(R)\ge \pd(R)$ and $[\tr(\omega_R)]_{\indeg(\mm)}$ contains a non-zerodivisor;
\item $[\omega_R]_{-a_R}$ contains a torsion-free element of $\omega_R$~(e.g., $R$ is a domain), $r_{\min}(R) \ge \pd(R)$ and $[\tr(\omega_R)]_{\indeg(\mm)}$ contains a non-zerodivisor;
\item $R$ is a Teter and level ring with $r(R)=\pd(R)$ and $\dim(R)>0$;
\item $R$ is Teter with  $\dim(R)>0$.
\end{enumerate}
Then \((1) \Rightarrow (2)\Rightarrow(3) \Rightarrow (4)\) holds. 
Moreover, if \(R\) is standard graded, all the conditions are equivalent; in particular, if any one of them holds, then \(r(R)=r_{\min}(R)=\pd(R)\).
\begin{proof}
Note that $\tr_R(\omega_R) \subset \mm$ because $R$ is not Gorenstein (see \cite[Lemma 2.1]{herzog2019trace}).

It follows from \autoref{rem:kakubeki?} that
$(1)\Rightarrow(2)$ holds.
We show (2)$\Rightarrow$(3).
$R$ is Teter by \autoref{p:Wednesday2}.
Since
$\operatorname{tr}(\omega_R)\ (\subset \mathfrak m)$ contains a non-zerodivisor, we have
$\dim(R)=\depth(R)>0$.
Hence it suffices to show that $R$ is level.
By \autoref{cor:gradedtype} we have $r(R)=\pd(R)$.
Therefore, from $\pd(R)=r(R)\ge r_{\min}(R)\ge \pd(R)$, we obtain
$r(R)=r_{\min}(R)=\pd(R)$, and thus $R$ is level.
The implication $(3)\Rightarrow(4)$ is clear.


Assume that $R$ is standard graded and we prove $(4)\Rightarrow(1)$.
By \autoref{cor:eeyan}, the ring $R$ is level, and by \autoref{cor:gradedtype} we have $r(R)=\pd(R)$.
\end{proof}
\end{theorem}

\begin{corollary}\label{cor:stdgradedTeter}
Assume that $R$ is non-Gorenstein standard graded ring with $\dim(R)>0$.
Then $R$ is Teter if and only if
$R$ is level, $r(R)=\pd(R)$ and $[\tr(\omega_R)]_1$ contains a non-zerodivisor.
\end{corollary}

\begin{question}\label{que:teretere}
Assume that $R$ is standard graded and $\dim(R)=0$.
By \autoref{rem:3.2}, \autoref{cor:gradedtype}, and \autoref{cor:eeyan}, if $R$ is Teter, then
it is nearly Gorenstein and level, and $r(R)=\pd(R)$ holds.
As a zero-dimensional analogue of \autoref{cor:stdgradedTeter}, we ask:
If $R$ is nearly Gorenstein and level but not Gorenstein, and if $r(R)=\pd(R)$, is $R$ necessarily Teter?
\end{question}

\begin{remark}\label{ex:unfortunately}
Unfortunately, \autoref{que:teretere} is not true in general.
Indeed, let
$S=\QQ[x_1,x_2,x_3,x_4]$ and
$$R=S/(x_{1}^{3},
x_{2}^{2},
x_{3}^{3},
x_{4}^{2},
x_{1}x_{2},  x_{1}x_{3}^{2}, x_{1}^{2}x_{3}, x_{2}x_{4},
x_{3}x_{4},
x_{2}x_{3}-x_{1}x_{4}).$$
Using \texttt{Macaulay2}~(\cite{M2}), one can check that $R$ is level, that $r(R)=\pd(R)=4$, and that
$\tr(\omega_R)=\mm_R$, whereas $R$ is not Teter.
Indeed, this can be verified by running the following computation in Macaulay~2. 
The graded minimal free resolution of $R$ over $S$, with degree shifts suppressed, is
\[
0 \longrightarrow 
S^{\oplus 4}
\xrightarrow{A} S^{\oplus 15}
\longrightarrow S^{\oplus 20}
\longrightarrow S^{\oplus 10}
\longrightarrow S \longrightarrow R \longrightarrow 0 .
\]
Hence we obtain $r(R)=\pd(R)=4$. 
By inspecting the Betti table of the minimal free resolution, we also see that $R$ is level. 
Moreover, $C=\ker(R\otimes_S A)$ has the following sixteen vectors as a basis.
\[
\left\{\!
\begin{array}{*{8}{c}}
\left[\!\begin{smallmatrix}0\\0\\0\\ x_{4}\end{smallmatrix}\!\right] &
\left[\!\begin{smallmatrix}0\\ x_{4}\\ x_{2}\\ x_{3}\end{smallmatrix}\!\right] &
\left[\!\begin{smallmatrix}0\\0\\0\\ x_{2}\end{smallmatrix}\!\right] &
\left[\!\begin{smallmatrix}-x_{4}\\ x_{2}\\ 0\\ x_{1}\end{smallmatrix}\!\right] &
\left[\!\begin{smallmatrix}x_{1}x_{4}\\ 0\\ 0\\ 0\end{smallmatrix}\!\right] &
\left[\!\begin{smallmatrix}x_{3}^{2}\\ 0\\ 0\\ 0\end{smallmatrix}\!\right] &
\left[\!\begin{smallmatrix}x_{1}x_{3}\\ 0\\ 0\\ 0\end{smallmatrix}\!\right] &
\left[\!\begin{smallmatrix}x_{1}^{2}\\ 0\\ 0\\ 0\end{smallmatrix}\!\right] \\
\left[\!\begin{smallmatrix}0\\ x_{1}x_{4}\\ 0\\ 0\end{smallmatrix}\!\right] &
\left[\!\begin{smallmatrix}0\\ x_{3}^{2}\\ 0\\ 0\end{smallmatrix}\!\right] &
\left[\!\begin{smallmatrix}0\\ x_{1}x_{3}\\ 0\\ 0\end{smallmatrix}\!\right] &
\left[\!\begin{smallmatrix}0\\ x_{1}^{2}\\ 0\\ 0\end{smallmatrix}\!\right] &
\left[\!\begin{smallmatrix}0\\ 0\\ x_{1}x_{4}\\ 0\end{smallmatrix}\!\right] &
\left[\!\begin{smallmatrix}0\\ 0\\ x_{3}^{2}\\ 0\end{smallmatrix}\!\right] &
\left[\!\begin{smallmatrix}0\\ 0\\ x_{1}x_{3}\\ 0\end{smallmatrix}\!\right] &
\left[\!\begin{smallmatrix}0\\ 0\\ x_{1}^{2}\\ 0\end{smallmatrix}\!\right]
\end{array}
\!\right\}
\]
Therefore, since the ideal generated by the $1$-minors of $C$ equals $\mm_R$, it follows from the graded version of \cite[Corollary~3.2]{herzog2019trace} that $\tr(\omega_R)=\mm_R$. On the other hand, $\mm_R$ cannot be generated solely by the $1$-minors of any single vector in $C$. Hence, by \cite[Corollary~3.8]{gasanova2022rings}, $R$ is not of Teter type.
\end{remark}


\section{Relations between the nearly Gorenstein and Teter properties}\label{sect4}
In this section, as an application of \autoref{cor:stdgradedTeter}, we mainly investigate the relationship between nearly Gorenstein rings and Teter rings. In particular, we show that various known nearly Gorenstein rings are Teter (\autoref{thm:examples}), and we clarify the relationship between nearly Gorenstein rings and the their type~(\autoref{cor:omgCMtype}).
Throughout this section, we retain \autoref{setup1} and use the notation of \autoref{section2}.
In addition, we assume throughout that $(R,\mathfrak m)$ is a positively graded Cohen--Macaulay ring.

It is known that every one-dimensional standard graded generically Gorenstein ring over an infinite field with minimal multiplicity is Teter (see the proof of \cite[Proposition~5.2]{puthenpurakal2025higher}). In dimension at least two, an analogous statement does not hold in general~(see \autoref{ex:mmm}).
In fact, we can characterize the Teter property in terms of the canonical trace as follows.

\begin{corollary}\label{cor:mmOMG}
Assume that $R$ is a non-Gorenstein standard graded ring which has a minimal multiplicity.
Then
$R$ is Teter with $\dim(R)>0$
if and only if
$[\tr_R(\omega_R)]_1$ contains a non-zerodivisor.
\end{corollary}
\begin{proof}
It is well known that if $R$ is standard graded and has minimal multiplicity, then it is a level ring and $r_{\min}(R)=h_1=\codim(R)$ (see \cite[Section~4]{eisenbud1984linear}). Hence this follows from \autoref{cor:stdgradedTeter}.
\end{proof}

\begin{remark}\label{ex:mmm}
As mentioned above, in dimension at least two a ring with minimal multiplicity need not be Teter. To see this, let $n \ge 4$ and let
$X = (x_{i,j})_{1 \le i \le 2,\; 1 \le j \le n}$
be a $2 \times n$ generic matrix. Consider the generic determinantal ring
$R = \kk[X]/I_2(X)$
defined by the $2 \times 2$ minors of $X$. Then $R$ is a Cohen--Macaulay standard graded domain; moreover, it is the homogeneous coordinate ring of a rational normal scroll and has minimal multiplicity (see \cite{bruns1998cohen,eisenbud1984linear}). On the other hand, by \cite[Theorem~1.1]{ficarra2024canonical!} we have
$\tr_R(\omega_R) = I_1(X)^{n-2} R = \mm_R^{n-2}$,
and in particular $[\tr_R(\omega_R)]_1 = (0)$, so it contains no non-zerodivisor of $R$. Thus, by \autoref{cor:mmOMG} $R$ is not Teter.
\end{remark}

We recall that for standard graded Teter rings, the difference between the type and the number of minimal generators of the canonical module in the smallest degree is always $0$.
In what follows, we first show that several known classes of nearly Gorenstein rings are Teter.
Moreover, we show that there exist semi-standard graded Teter rings for which the above difference can take an arbitrary value.


\begin{corollary}[{see \autoref{thm:examplesX}}]\label{thm:examples}
Assume that \(R\) is nearly Gorenstein, but not Gorenstein.
Assume further that 
$R$ satisfies one of the following conditions:
\begin{enumerate}[\rm (1)]
\item \(R\) is a level semi-standard graded ring with \(r(R)=\pd(R)\) and $\dim(R)>0$;
\item \(R\) is a semi-standard graded domain with \(\pd(R)=2\) and \(\dim(R)\ge 2\);
\item \(R\) is a standard graded affine semigroup ring with $\pd(R) \le 3$ and $\dim(R)=2$;
\item \(R\) is a standard graded ring which has minimal multiplicity;
\item \(R\) is a Stanley--Reisner ring;
\item \(R\) is a standard graded domain whose $h$-vector is $h(R)=(1,a,a)$ for some $a>0$;
\item \(R\) is
a
non-level
semi-standard graded 2-dimensional affine semigroup ring whose $h$-vector is $h(R)=(1,a,a+1)$ for some $a>0$.
\end{enumerate}
Then 
$R$ is Teter.
Moreover, for any positive integer $a>0$, there exists a nearly Gorenstein (non-standard graded) semi-standard graded Teter ring $R$ satisfying $(7)$ such that $r(R)-r_{\min}(R)=a$.
\begin{proof}
(1):
Since $R$ is semi-standard graded and $\dim(R)>0$, we have $R_{\indeg(\mm_R)}=R_1$, which contains a non-zerodivisor of $R$.
Hence, by \autoref{cor:nGfunnyTeter}\,(1), $R$ is Teter.

(2):
Since $R$ is nearly Gorenstein, in particular $[\tr_R(\omega_R)]_1\neq(0)$.
Therefore, by \cite[Theorem 2.1]{ficarra2025canonical}, we have $r(R)=2$.
Thus $R$ satisfies (1), and hence is Teter.

(3):
Since $R$ is not Gorenstein, we have $\pd(R)=2$ or $\pd(R)=3$.
If $\pd(R)=2$, then $R$ satisfies (2) and hence is Teter.
Therefore we may assume $\pd(R)=3$.
As $R$ is a two-dimensional standard graded affine semigroup ring, note that it is isomorphic to a projective monomial curve.
By \cite[Theorem A]{miyashita2023nearly}, in this case $R$ is level and $r(R)=\pd(R)=3$.
Hence $R$ satisfies (1), so $R$ is Teter.

(4) and (5):
(4) follows from \autoref{cor:mmOMG}.
By \cite[Theorem A]{miyashita2025canonical}, $R$ is isomorphic to the Stanley--Reisner ring corresponding to a path of length at least $3$, and in particular has minimal multiplicity.
Hence $R$ satisfies (4), and therefore is Teter.

(6):
In this case it is known that $R$ is level (see \cite[Corollary 3.11]{yanagawa1995castelnuovo} or \cite[Corollary 4.3]{higashitani2016almost}).
Moreover, since $R$ is standard graded, note that $\pd(R)=h_1=a$; from the shape of the $h$-vector it follows that $r(R)=a=\pd(R)$.
Hence $R$ satisfies (1), and thus is Teter.

(7):
By \cite[Proposition 5.4 and Theorem 6.3]{miyashita2024comparing}, $R$ is non-level and is a nearly Gorenstein semi-standard graded ring in the sense of \cite{goto2015almost}.
Hence, by \cite[Theorem 6.3]{miyashita2024comparing}, we may write $R=\kk[S]$, where
\[
S \cong \big\langle \{(2i,\,2a+2-2i): 0\le i\le a+1\} \cup \{(2j+2k-1,\,4a-2j-2k+5): 0\le j\le a\} \big\rangle
\]
for some
$1\le k\le a+2$.
By the proof of \cite[Proposition 3.9]{miyashita2024comparing},
we have an isomorphism
$$\omega_R \cong
(t^{\mathbf{s}} : s \in \omega_S)R \subset Q(R),$$
where
$\omega_S= \big\langle \{\, (1-2i,\,-1+2i) : k-1-a \le i \le k-1 \,\} \cup \{\, (2i,\,2a+2-i) : 1 \le i \le a \,\} \big\rangle$.
Thus, for the ideal $I = t^{(2k-3,\,2a+5-2k)} \omega_R \cong \omega_R$, 
since $I \subset R$ and $\operatorname{codim}(R/I) = 2$, the ring $R$ is Teter.
Moreover, since $r_{\min}(R) = a+1$ and $r(R) = 2a+1$, we have $r(R) - r_{\min}(R) = a$.
\end{proof}
\end{corollary}

Next, we investigate the relationship between the trace ideal of the canonical module and the type.

\begin{proposition}\label{cor:nGfunnyTeter}
Assume that
$[\tr_R(\omega_R)]_{\indeg(\mm_R)}$ contains  non-zero divisor of $R$
(e.g., $R$ is nearly Gorenstein and that $R_{\indeg(\mm_R)}$ contains a non-zero divisor of $R$).
Then the following hold:
\begin{itemize}
\item[\rm (1)] If $r(R)\ge \pd(R)$ and $R$ is level,
then it is Teter and $r(R)=\pd(R)$;
\item[\rm (2)] If $r_{\min}(R)\ge \pd(R)$,
then it is level and $r(R)=\pd(R)$.
\end{itemize}
\begin{proof}
(1) follows from the implication $(1)\Rightarrow(3)$ in \autoref{thm:stdgradedTeter}.
(2) follows from the implication $(2)\Rightarrow(3)$ in \autoref{thm:stdgradedTeter}.
\end{proof}
\end{proposition}

\begin{remark}
If, in \autoref{cor:nGfunnyTeter}, we drop only the assumption that $R_{\indeg(\mm_R)}$ contains a non-zerodivisor, the assertion analogous to (1) does not hold. Indeed, the nearly Gorenstein Artin ring $R$ in \autoref{ex:unfortunately} satisfies $r(R)=\pd(R)=4$ and is level, yet it is not Teter.
\end{remark}

As a consequence, we obtain the following corollary concerning the type of nearly Gorenstein rings.

\begin{corollary}[{\autoref{cor:omgCMtypeX}}]\label{cor:omgCMtype}
Assume $R$ is non-Gorenstein nearly Gorenstein, and that
$R_{\indeg(\mm_R)}$ contains a non-zero divisor of $R$~(e.g., when $R$ is a domain or a semi-standard graded ring with $\dim(R) > 0$).
Then the following hold:
\begin{itemize}
\item[\rm (1)]
If $R$ is level,
then $r(R) \le \pd(R)$;
\item[\rm (2)]
If $r_{\min}(R) \ge \pd(R)$,
then $R$ is level and $r(R)=\pd(R)$;
\item[\rm (3)]
If $R$ is standard graded, $\pd(R)=3$, $\dim(R) \ge 2$ and $r_{\min}(R) \neq 2$, then $R$ is level and $r(R)=3$;
\end{itemize}
\begin{proof}
(1) and (2): Assume that $r(R) > \pd(R)$. Then, by \autoref{cor:nGfunnyTeter}~(1), we obtain $r(R) = \pd(R)$, a contradiction. Hence $r(R) \le \pd(R)$.
(2) follows from \autoref{cor:nGfunnyTeter}~(2).

(3):
Suppose $r_{\min}(R)=1$. Then $R$ is pseudo-Gorenstein. Hence, by \cite[Corollary 4.7(1)]{miyashita2025pseudo}, $R$ is Gorenstein, a contradiction. Therefore $r_{\min}(R)\ge 3$, and in this case (2) yields $r_{\min}(R)=r(R)=3$.
\end{proof}
\end{corollary}

\begin{corollary}\label{cor:fine2}
If $R$ is a nearly Gorenstein and level domain, then $r(R) \le \pd(R)$.
\end{corollary}
\begin{proof}
This follows from \autoref{cor:omgCMtypeX}~(1).
\end{proof}

\begin{remark}\label{rem:lolol}
\autoref{cor:omgCMtype}~(3) fails when $\dim(R)<2$.
More precisely, when $\dim(R)=0$ there exists a standard graded nearly Gorenstein ring $R$ with
$\pd(R)=3$ and $r_{\min}(R)=1$, yet $r(R)=2$; in particular $R$ is not level.
Indeed, set
\[
A=\QQ[x,y,z]/(xz^2 - y^3,\ x^3 + xy^2 - y^2 z,\ x^2 y + y^3 - z^3),
\]
where $\deg(x)=\deg(y)=\deg(z)=1$.
Then $A$ is a non-level nearly Gorenstein standard graded domain with $\pd(A)=r(A)=2$ (see \cite[Example 3.1]{miyashita2024levelness}).
Hence $x^2+z^2\in \mm_A^2$ is a non-zerodivisor on $A$, and by \cite[Proposition 2.3\,(b)]{herzog2019trace} the ring
$R=A/(x^2+z^2)A$
is a non-level nearly Gorenstein standard graded ring with $\pd(R)=3$ and $r(R)=2$.
\end{remark}

\section{The Teter property for special constructions of algebras}\label{sect5}
In this section, as an application of \autoref{thm:stdgradedTeter}, we mainly discuss the Teter property of fiber products, tensor products, Segre products, and certain Veronese subalgebras of standard graded rings.

Throughout this section, we retain \autoref{setup1} and use the notation of \autoref{section2}.

\subsection{The Teter property of fiber products of standard graded rings}
The aim of this subsection is to study the Teter property of fiber products of Cohen–Macaulay standard graded rings.
Let $A$, $B$ be positively graded Noetherian rings
and let $\kk$ be a field.
In what follows, we suppose the following.

\begin{setup}\label{conditiona}
 \( \kk = A_0 = B_0 \) and the maps \( f \) and \( g \) are the canonical graded surjections 
\begin{center}
$f:A\to A/A_{>0}\cong \kk$ \quad and \quad $g: B\to B/B_{>0}\cong \kk$.
\end{center}
\end{setup}

\begin{definition}\label{def:fiberrrrrrrrrr}
Assume \autoref{conditiona}.
The subring of $A\times B$ 
$$
R := A \times_T B = \{(a, b) \in A \times B : f(a) = g(b)\}
$$
is called {\it the fiber product of $A$ and $B$ over $\kk$}.
\end{definition}

Then, one can endow $R$ with a natural graded structure by $R_n:=\{(a,b)\in A_n\times B_n : f(a)=g(b)\}$. 
By definition of the fiber product, we get a graded exact sequence
\begin{align}\label{eq0}
0 \longrightarrow R \overset{\iota}{\longrightarrow} A \oplus B \overset{\phi}{\longrightarrow} \kk
{\longrightarrow} 0
\end{align}
of $R$-modules, where $\phi = \left(\begin{smallmatrix}
f \\
-g
\end{smallmatrix}\right)$.
The map $\phi$ is surjective because $f$ or $g$ is surjective.


\begin{remark}[{see \cite{lescot2006serie} or \cite[Remark 3.1]{christensen2010growth}}]\label{rem:12}
It is known that the following equalities hold:
\[
  \dim(R)=\max\{\dim(S),\dim(T)\},\qquad
  \depth(R)=\min\{\depth(S), \depth(T), 1\}.
\]
In particular, $R$ is Cohen--Macaulay if and only if
$S$ and $T$ are Cohen--Macaulay and $\dim(S)=\dim(T)\le 1$.
\end{remark}

\begin{remark}[{see \cite[Definition 3.9, Remark 3.10]{kumashiro2025canonical}}]
For a Cohen--Macaulay graded ring $(S,\MM_S)$ possessing the graded canonical module $\omega_S$,
we set
\[
\tr_S^\dagger(\omega_S) :=
\begin{cases}
\tr_S(\omega_S) & \text{if } S \text{ is not Gorenstein}, \\
\mm_S & \text{if } S \text{ is Gorenstein.}
\end{cases}
\]
\end{remark}

The following may be well known to experts, but we include a proof for the reader's convenience.

\begin{lemma}\label{lem:fiberrrrrr}
Assume that $A$ and $B$ are generically Gorenstein Cohen--Macaulay graded rings with $\dim(A)=\dim(B)=1$.
Furthermore, assume that at least one of $A$ and $B$ is not regular.
Then the following hold:
\begin{enumerate}[\rm (1)]
\item
If neither $A$ nor $B$ is regular, then $r(R)=r(A)+r(B)+1$.
If $A$ is regular, then $r(R)=r(B)+1$;
\item
If neither $A$ nor $B$ is regular, then
$\pd(R)=\pd(A)+\pd(B)+1$.
If $A$ is regular, then $\pd(R)=\pd(B)+1$;
\item
$R$ is level if and only if one of the following holds:
\begin{enumerate}[\rm (i)]
\item $A$ and $B$ are both level with $a_A=a_B=0$;
\item $A$ is regular and $B$ is level with $a_B=0$;
\item $B$ is regular and $A$ is level with $a_A=0$.
\end{enumerate}
\item
The condition that $[\tr_R(\omega_R)]_{\indeg(\mm_R)}$ contains a non-zerodivisor
is equivalent to the condition that
$\indeg(\mm_A)=\indeg(\mm_B)$ and both
$[\tr_A(\omega_A)]_{\indeg(\mm_A)}$ and $[\tr_B(\omega_B)]_{\indeg(\mm_B)}$
contain a non-zerodivisor.
\end{enumerate}
\end{lemma}

\begin{proof}
(1):
It follows from the proof of \cite[Remark 3.18]{kumashiro2025canonical} that
$R_{\mm_R} \cong A_{\mm_A} \times_\kk B_{\mm_B}$.
Moreover, by the proof of \cite[Fact 2.6]{nasseh2019applications},
if neither $A$ nor $B$ is regular, then
$r(R_{\mm_R})=r(A_{\mm_A})+r(B_{\mm_B})+1$,
whereas if $A$ is regular, then
$r(R_{\mm_R})=r(B_{\mm_B})+1$.
Since $r(R_{\mm_R})=r(R)$, $r(A_{\mm_A})=r(A)$, and
$r(B_{\mm_B})=r(B)$, the desired assertion follows.

(2):
If neither $A$ nor $B$ is regular, then applying the Auslander--Buchsbaum formula
to $R$, $A$, and $B$ respectively, we obtain
$$\pd(R)=\embdim(A)+\embdim(B)-1=(\pd(A)+1)+(\pd(B)+1)-1=\pd(A)+\pd(B)+1.$$
Similarly, when $A$ is regular we obtain
$\pd(R)=\pd(B)+1$.

(3):
By \cite[Proposition 2.18~(2)]{kumashiro2025canonical} there is a graded exact sequence
of degree $0$:
\[
0 \rightarrow
\omega_{A} \oplus \omega_{B} \xrightarrow{\iota} \omega_R
\xrightarrow{\pi} \kk \rightarrow 0.
\]
Let $f_1,\dots,f_n$ and $g_1,\dots,g_m$ be minimal generators of $\omega_A$ and $\omega_B$
as $R$-modules, respectively, and choose a homogeneous element
$x \in [\omega_R]_0$ with $\pi(x)=1$.
Then the above short exact sequence shows that
$\{\iota(f_1),\dots,\iota(f_n),\iota(g_1),\dots,\iota(g_m),x\}$ is a generating set of
$\omega_R$ as an $R$-module.
If neither $A$ nor $B$ is regular, then by (1) this is a minimal generating set of $\omega_R$. Hence $R$ is level if and only if {\rm(i)} holds.

Now assume that $A$ is regular.
We first show that $x$ cannot be removed from a minimal generating set.
If $x$ were contained in the submodule generated by
$\{\iota(f_1),\iota(g_1),\dots,\iota(g_m)\}$, then there would exist
$a,b_1,\dots,b_m \in R$ such that
$x = a\,\iota(f_1) + \sum_{j=1}^m b_j\,\iota(g_j)$.
Applying $\pi$ to both sides, we obtain
$1 = \pi(x)
  = a\,\pi(\iota(f_1)) + \sum_{j=1}^m b_j\,\pi(\iota(g_j))
  = 0$,
which is a contradiction (here we use $\pi\circ\iota=0$).
Hence $x$ must appear in every minimal generating set of $\omega_R$.

Next we show that none of the elements $\iota(g_j)$ can be removed from a minimal generating set.
Without loss of generality, we may assume $j=1$.
Suppose that
there exist $a,b_2,\dots,b_m,c \in R$ such that
$\iota(g_1)
= a\,\iota(f_1) + \sum_{j=2}^m b_j\,\iota(g_j) + c\,x$.
Applying $\pi$ to both sides, we obtain
$0 = \pi(\iota(g_1))
  = a\,\pi(\iota(f_1)) + \sum_{j=2}^m b_j\,\pi(\iota(g_j)) + c\,\pi(x)
  = c$,
and hence $c=0$.
Thus
$\iota(g_1) = a\,\iota(f_1) + \sum_{j=2}^m b_j\,\iota(g_j)$.
Since $\iota$ is injective, this can be rewritten in $\omega_A \oplus \omega_B$ as
$(0,g_1)
= a\,(f_1,0) + \sum_{j=2}^m b_j\,(0,g_j)$.
Using the identification $R_{\mathfrak m_R} \cong A_{\mathfrak m_A} \times_k B_{\mathfrak m_B}$,
we can write $a=(a_A,a_B)$ and $b_j=(b_{j,A},b_{j,B})$.
Then the right-hand side becomes
$(a_A f_1,\ \sum_{j=2}^m b_{j,B} g_j)$,
and comparing the second components, we obtain
$g_1 = \sum_{j=2}^m b_{j,B} g_j$.
This contradicts the assumption that $\{g_1,\dots,g_m\}$ is a minimal generating set of $\omega_B$.
Thus, for every $j$, the element $\iota(g_j)$ cannot be removed from a minimal generating set.

Therefore, in the generating set
$\{\iota(f_1),\iota(g_1),\dots,\iota(g_m),x\}$ of the $R$-module $\omega_R$,
neither $x$ nor any $\iota(g_j)$ can be removed from a minimal generating set.
By (1) we have $r(R)=m+1$, and hence
$\{\iota(g_1),\dots,\iota(g_m),x\}$
is a minimal generating set of $\omega_R$ as an $R$-module.
In particular, $R$ is level if and only if {\rm (ii)} holds.
If $B$ is regular, the argument is analogous: then $m=1$ and
$\{\iota(f_1),\dots,\iota(f_n),x\}$ is a minimal generating set of $\omega_R$ as an $R$-module,
and $R$ is level if and only if {\rm (iii)} holds.

(4):
By \cite[Remark 3.18]{kumashiro2025canonical} we have
$\operatorname{tr}_R(\omega_R) = \operatorname{tr}_A^{\dagger}(\omega_A)R \oplus \operatorname{tr}_B^{\dagger}(\omega_B)R$.
Hence, for a homogeneous element $x=(a,b) \in [\tr_R(\omega_R)]_{\indeg(\mm_R)}$,
the condition that $x$ is a non-zerodivisor of $R$ is equivalent to the condition that
$\indeg(\mm_A)=\indeg(\mm_B)=\indeg(\mm_R)$ and that
$a \in [\tr_A(\omega_A)]_{\indeg(\mm_A)}$ is a non-zerodivisor of $A$ and
$b \in [\tr_B(\omega_B)]_{\indeg(\mm_B)}$ is a non-zerodivisor of $B$.
This proves the assertion.
\end{proof}

\begin{remark}\label{remark:abcdefg}
Assume that $R$ is 1-dimensional Cohen--Macaulay generically Gorenstein standard graded ring.
If $R$ has minimal multiplicity,
then it is Teter.
\begin{proof}
We may assume that $R_0$ is infinite.
In this case, in the proof of \cite[Proposition 5.2]{puthenpurakal2025higher}, the condition (3) in \autoref{rem:3.2}\,(c) is verified.
By \autoref{rem:3.2}\,(c), this is equivalent to saying that $R$ is Teter.
\end{proof}
\end{remark}

\begin{corollary}[{see \autoref{thm:examplesX}}]\label{thm:Fiber2}
Assume that $A$ and $B$ are Cohen--Macaulay generically Gorenstein standard graded rings with $\dim(A)=\dim(B)=1$, not both of which are regular.
Then $R$ is Teter
if and only if both of $A$ and $B$ have minimal multiplicity.
\begin{proof}
Assume that $R$ is Teter.
By \autoref{thm:stdgradedTeterX} it follows that $R$ is level.
By \autoref{lem:fiberrrrrr}~(3), we have $a_A=a_B=0$.
Since $A$ and $B$ are one-dimensional standard graded rings, this means that both $A$ and $B$ have minimal multiplicity.

Next assume that both of $A$ and $B$ have minimal multiplicity.
We may assume that $A_0=B_0=\kk$ is an infinite field.
Under this assumption, together with \autoref{remark:abcdefg}, the rings $A$ and $B$ are each either Gorenstein or Teter, and moreover we have $a_A=a_B=0$ and both $A$ and $B$ are level (see \cite[Section 4]{eisenbud1984linear}).
In the Gorenstein case, the hypothesis further implies that they are hypersurfaces.
Combining this with \autoref{thm:stdgradedTeter}, we see that in either case, hypersurface or Teter, the equalities $r(A)=\pd(A)$ and $r(B)=\pd(B)$ hold, and both $[\tr_A(\omega_A)]_1$ and $[\tr_B(\omega_B)]_1$ contain a non-zerodivisor.
Therefore, by \autoref{lem:fiberrrrrr}, the ring $R$ is Teter.
\end{proof}
\end{corollary}

\subsection{The Teter property of Veronese subalgebras}
The purpose of this subsection is to investigate the Teter property of suitably low-dimensional Veronese subalgebras.
The next statement is well known to experts, but we include a proof for the reader's convenience.

\begin{remark}\label{rem:trivial1}
Assume $R$
is standard graded and $1 \le \dim(R)\le2$.
If $R$ has minimal multiplicity
(i.e. $e(R)=\embdim(R)-\dim(R)+1$),
then
$R^{(k)}$
also has minimal multiplicity and $e(R^{(k)})=k^{\dim(R)-1}e(R)$ for all $k \in \ZZ_{>0}$.
In particular, $R^{(k)}$ is Gorenstein if and only if $k^{\dim(R)-1}e(R) \le 2$.
\end{remark}
\begin{proof}
Set $d=\dim(R)\le2$.
If $R$ has minimal multiplicity and $d\le2$, then its Hilbert series is
\[
\HS_R(t)=\frac{1+ct}{(1-t)^d}
\]
for some $c \in \NN$.
Consequently, we have
$e(R)=1+c$
and
\[
\dim_\kk R_n=
\begin{cases}
1+c & (d=1,\ n\ge1),\\[2pt]
(c+1)\,n+1 & (d=2,\ n\ge0).
\end{cases}
\]
Moreover, for every $k\ge1$ we have:
$\mu\!\bigl(\mm^{(k)}\bigr)
= \dim_\kk \bigl[R^{(k)}\bigr]_1
 = \dim_\kk R_k$
and
$e\!\bigl(R^{(k)}\bigr)
= k^{d-1}e(R)$.
\begin{itemize}
\item When $d=2$,
since $\dim_\kk R_k=(c+1)k+1$ and $e(R)=c+1$,
we have
\[
e\bigl(R^{(k)}\bigr)=ke(R)=k(c+1)
=\mu\!\bigl(\mm^{(k)}\bigr)-d+1.
\]
\item When $d=1$,
since $\dim_\kk R_k=1+c=e(R)$ for all $k\ge1$,
we have
\[
e\bigl(R^{(k)}\bigr)=e(R)=1+c
=\mu \bigl(\mm^{(k)}\bigr)-d+1.
\]
\end{itemize}
Combining the three cases yields the assertion.
Noting that $R$ is level and that $k^{d-1}e(R)=e(R^{(k)})=1+h_1$,
where $h(R^{(k)})=(1,h_1)$ is the $h$-vector of $R^{(k)}$ (here we also allow $h_1=0$).
Thus we see that
$R^{(k)}$ is Gorenstein if and only if $h_1\le 1$, equivalently $k^{d-1}e(R)\le 2$.
\end{proof}

\begin{theorem}\label{thm:Veronese1}
Assume that $R$ is standard graded and $\dim(R) \le 2$.
If $R$ has minimal multiplicity and that it is either Teter or Gorenstein,
then 
$R^{(k)}$ is again Teter or Gorenstein for every $k \ge 1$.
In particular, if either $\dim(R)=1$ with $e(R)\ge 3$ or $\dim(R)=2$, then $R^{(k)}$ is Teter for every $k\ge 2$.
\end{theorem}
\begin{proof}
When $\dim(R)=0$,
note that $R=\kk\oplus R_1$ because $R$ has minimal multiplicity.
Thus $R^{(k)}=\kk$ is Gorenstein for $k \ge 2$; for $k=1$ the claim is tautological.
Thus we may assume $\dim(R)=1$ or $\dim(R)=2$.

Since $R$ is either Teter or Gorenstein, it is generically Gorenstein, so $\omega_R$ can be identified with a fractional ideal of $R$.
By \autoref{thm:stdgradedTeter} there exists a non-zerodivisor $f\in[\tr(\omega_R)]_1$.
By \cite[Theorem~3.4]{miyashita2024linear} we have
$\tr_{R^{(k)}}(\omega_R^{(k)}) \supset (\mm^{k-1}(f))^{(k)}$.
In particular, for the degree-one non-zerodivisor $f^k$ of $R^{(k)}$ we obtain
$f^k \in [\tr_{R^{(k)}}(\omega_R^{(k)})]_1$.
Hence, by \autoref{cor:mmOMG} and \autoref{rem:trivial1},  $R^{(k)}$ is either Gorenstein or Teter.
In particular, if either $\dim(R)=1$ with $e(R)\ge 3$ or $\dim(R)=2$, then for every $k\ge 2$ we have $k^{\dim(R)-1}e(R)\ge 3$, and hence $R^{(k)}$ is not Gorenstein by \autoref{rem:trivial1}. Therefore $R^{(k)}$ is Teter.
\end{proof}


\begin{corollary}[{see \autoref{thm:examplesX}}]\label{thm:Veronese2}
Assume $R$ is a non-Gorenstein generically Gorenstein standard graded ring with minimal multiplicity.
If either $\dim(R)=1$ or $\tr_R(\omega_R)=\mm_R$ with $\dim(R)=2$, then $R^{(k)}$ is Teter for every $k\ge 2$.
\end{corollary}
\begin{proof}
By \autoref{cor:mmOMG} and \autoref{remark:abcdefg}, $R$ is Teter under our assumptions. Hence, by \autoref{thm:Veronese1}, $R^{(k)}$ is Teter for every $k\ge 2$.
\end{proof}

\section{Teter numerical semigroup rings}\label{sect6}
In the \autoref{sect4} and \autoref{sect5}, we primarily investigated the Teter property of standard graded rings.
In this section, we investigate the Teter property of numerical semigroup rings, which serve as typical examples of rings that are not semi-standard graded. Our goal is to completely characterize the Teter property of numerical semigroup rings in terms of numerical semigroups. Furthermore, using this characterization, we discuss the relationships with almost Gorenstein and nearly Gorenstein rings and some examples.

Throughout this section, we retain \autoref{setup1} and use the notation of
Section 2 (see \autoref{subsec:nume}).
We also assume that $R_H$ is not Gorenstein.
Applying the definition of Teter to a numerical semigroup ring, it can be rephrased as follows.
\begin{remark}\label{numerical_rem1}
    Let $H$ be a numerical semigroup.
    The following are equivalent:
    \begin{enumerate}[\rm (1)]
        \item $R_H$ is Teter;
        \item $\omega_{R_H}\cong I$ and $\embdim(R_H/I)=\mu((\bar{t^{a_1}},\ldots,\bar{t^{a_n}})+I)\leq 1$ for some ideal $I\subsetneq R_H$;
        \item $t^{\gamma}\omega_{R_H}\subsetneq {R_H}$ and for at least $n-1$ indices $i\in [n]$, $t^{a_i}\in t^{\gamma}\omega_{R_H}$ for some $\gamma\in\mathbb{Z}_{>0}$.
    \end{enumerate}
\end{remark}
\begin{proof}
    $\rm(1) \Leftrightarrow \rm(2)$ follows from \autoref{rem:3.2}~(c).
    We show that (2) $\Rightarrow$ (3).
    Since there exists an isomorphism $\omega_{R_H} \cong I$ as graded $R_H$-modules and $\omega_{R_H}$ is a fractional ideal, there exists $\gamma \in \mathbb{Z}$ such that $t^{\gamma}\omega_{R_H} = I$.
    Here, since $\omega_{R_H} = \sum_{\alpha \in \PF(H)} t^{-\alpha}R_H$, we can choose $\gamma > 0$.
    Let $V = \mathfrak{m}/\mathfrak{m}^2 (= \sum_{i=1}^{n} k\overline{t^{a_i}})$ and $W = (I + \mathfrak{m}^2)/\mathfrak{m}^2$.
    Then the condition $\codim(R_H/I) \leq 1$ implies $\dim_k(V/W) = \dim_k(\mathfrak{m}/(I + \mathfrak{m}^2)) \leq 1$, which is equivalent to saying that the dimension of $W$ as a subspace of $V$ is at least $n-1$.
    Noting that $I$ is a monomial ideal, we see that $I$ contains at least $n-1$ elements of the form $t^{a_i}$.
    (3) $\Rightarrow$ (2). is clear.
\end{proof}

Regarding the condition $(3)$ of \autoref{numerical_rem1}, we define the following:
\begin{definition}
    For a numerical semigroup $H$, we define the following sets:
    \begin{enumerate}[\rm (1)]
        \item $T_1(H):=\{\gamma\in\mathbb{Z}_{>0}\,|\,t^{\gamma}\omega_{R_H}\subsetneq {R_H}\}$;
        \item $T_2(H):=\{\gamma\in\mathbb{Z}_{>0}\,|\,\text{for at least } n-1 \text{ indices } i\in [n], t^{a_i}\in t^{\gamma}\omega_{R_H}\}$;
        \item $T(H):=T_1(H)\cap T_2(H)$;
        \item Let $\gamma_{H}:=\min(T(H))\geq\frob(H)+a_1$ and $\delta_{H}:=\gamma_{H}-(\frob(H)+a_1)\geq0$.
    \end{enumerate}
    In other words, $R_H$ is Teter if and only if $T(H)\not=\emptyset$.
\end{definition}

\begin{proposition}\label{numerical_prop1}
    The following hold:
    \begin{enumerate}[\rm (1)]
        \item If $\gamma\in T_1(H)$, then $\frob(H)+a_1\leq \gamma$;
        \item If $R_H$ is Teter, then 
        $\delta_H>0$ if and only if $t^{a_1}\not\in t^{\gamma_H}\omega_{R_H}.$
    \end{enumerate}
\end{proposition}
\begin{proof}
(1): Assume $\gamma\in T_1(H)$ and $\frob(H)+a_1> \gamma$.
        By definition, $\gamma-\frob(H)\in H$. Since $\gamma-\frob(H)<a_1$, we must have $\gamma=\frob(H)$.
        Since $R_H$ is not Gorenstein, there exists a pseudo-Frobenius number $\alpha$ different from $\frob(H)$. However, this implies $\gamma-\alpha\in H$, which is a contradiction.
        
(2): If $\delta_H>0$, we have $\gamma_H-\frob(H)=\frob(H)+a_1+\delta_H-\frob(H)>a_1$, which implies $t^{a_1}\not\in t^{\gamma_H}\omega_{R_H}$.
        Conversely, if $\delta_H=0$, we have $\gamma_H-\frob(H)=
        a_1$, which implies $t^{a_1}\in t^{\gamma_H}\omega_{R_H}$.
\end{proof}

To prove the theorem, we show the following lemma.

\begin{lemma}\label{numerical_lemma1}
    Assume $\gamma\in T_1(H)$ for a numerical semigroup $H$.
    Then, $t^{a_i} \in t^{\gamma}\omega_{R_H}$ for some $i\in [n]$ if and only if there exists a unique $i_0\in [\type(H)]$ such that $a_i=\gamma-\alpha_{i_0}$.
\end{lemma}
\begin{proof}
    The direction $\Leftarrow$ is clear.
    We show $\Rightarrow$.
    Assume $t^{a_i}\in t^{\gamma}\omega_{R_H}$ for $i\in [n]$.
    Then there exists $i_0\in [\type(H)]$ such that $t^{a_i}\in t^{\gamma-\alpha_{i_0}}R_H$.
    Since $t^{\gamma-\alpha_{i_0}}R_H=\sum_{h\in H}t^{\gamma-\alpha_{i_0}+h}$, there exists $h\in H$ such that $a_i=\gamma-\alpha_{i_0}+h$.
    Since $\gamma\in T_1(H)$, we must have $\gamma-\alpha_i\in H$ for any $i\in [\type(H)]$.
    If $\gamma-\alpha_{i_0}=0$, it contradicts the fact that $\frob(H)<\gamma$ from \autoref{numerical_prop1}~(1), so $\gamma-\alpha_{i_0}\in H\setminus\{0\}$.
    Therefore, by the minimality of the generator $a_i$, we obtain $h=0$, and thus $a_i=\gamma-\alpha_{i_0}$.
    Moreover, if there exists $\alpha_{i_1}$ and $\alpha_{i_2}$ which satisfies $a_i=\gamma-\alpha_{i_1}=\gamma-\alpha_{i_2}$, then $\alpha_{i_1}=\alpha_{i_2}$.
    Thus such $i_0$ is unique.
\end{proof}

\begin{theorem}\label{thm:numeOZAKI}
    Let $H$ be a numerical semigroup. Then the following are equivalent:
    \begin{enumerate}[\rm (1)]
        \item $R_H$ is Teter;
        \item $\PF(H)=\{\alpha_1,\ldots,\alpha_{n-1}\}$ where  $\alpha_1<\ldots<\alpha_{n-1}=\frob(H)$, and the following holds:\\
        There exists $s\in [n]$ such that for all $j$ with $2\leq j\leq n$,
        \begin{equation*}
        \frob(H)+a_1+\delta =
            \left\{ \,
                \begin{array}{cc}
                a_j+\alpha_{n-j} & (j<s); \\
                a_j+\alpha_{n+1-j} & (j>s),
                \end{array}
            \right.
        \end{equation*}
        where
        \begin{equation*}
        \delta =
            \left\{ \,
                \begin{array}{cc}
                a_2-a_1 & (s=1); \\
                0 & (s\not=1).
                \end{array}
            \right.
        \end{equation*}
    \end{enumerate}
    In particular, when $R_H$ is Teter, the only possible values for $\delta_H$ are $0$ or $a_2-a_1$.
\end{theorem}
\begin{proof}
(1)$\Rightarrow$ (2):
From \autoref{cor:gradedtype}, $\type(H)=\emb(H)-1$.
        Therefore, $\PF(H)$ consists of $n-1$ elements, say $\alpha_1<\ldots<\alpha_{n-1}=\frob(H)$.
        From \autoref{numerical_rem1} and \autoref{numerical_lemma1},
        The ring $R_H$ is Teter if and only if there exists a positive integer 
        $\gamma$ such that $t^{\gamma}\omega_R \subsetneq R$ and $t^{a_i} \in t^{\gamma}\omega_{R_H}$ 
        holds for exactly $n-1$ indices $i \in [n]$.
        Let $t^{a_s}$ be the element not contained in $t^{\gamma_H}\omega_R$,
        where $s\in[n]$.
        If $s=1$, then, by the \autoref{numerical_lemma1} and the order relation between $a_i$ and $\alpha_i$,
we have $a_j = \gamma_H - \alpha_{n+1-j}$ for every integer $j$ with $2 \le j \le n$.
In particular, taking $j=2$ gives $a_2 = \gamma_H - \alpha_{n-1}$, hence
$\delta_H = a_2 - a_1$.
If $s \neq 1$, then, by the \autoref{numerical_prop1}, $\delta_H = 0$, and the claim follows as in the case $s=1$.

(2)$\Rightarrow$ (1):
When $s=1$, considering $t^{\frob(H)+a_2}\omega_{R_H}$ shows that $R_H$ is Teter.
        When $s\not=1$, considering $t^{\frob(H)+a_1}\omega_{R_H}$ shows that $R_H$ is Teter.
\end{proof}
\begin{corollary}\label{numerical_cor1}
    When $H$ is a numerical semigroup with minimal multiplicity.
    The following are equivalent:
    \begin{enumerate}[\rm (1)]
        \item $R_H$ is Teter;
        \item $\PF(H)=\{\alpha_1,\ldots,\alpha_{n-1}\}$ where  $\alpha_1<\ldots<\alpha_{n-1}=\frob(H)$, and the following holds:\\ There exists $s\in \{1,n\}$ such that for all $j$ with $2\leq j\leq n$,
        \begin{equation*}
        \frob(H)+\delta =
            \left\{ \,
                \begin{array}{cc}
                \alpha_{j-1}+\alpha_{n-j} & (j<s); \\
                \alpha_{j-1}+\alpha_{n+1-j} & (j>s);
                \end{array}
            \right.
        \end{equation*}
        where
        \begin{equation*}
        \delta =
            \left\{ \,
                \begin{array}{cc}
                \alpha_1 & (s=1); \\
                0 & (s\not=1).
                \end{array}
            \right.
        \end{equation*}
        \item[\rm(3)]  $\PF(H)=\{\alpha_1,\ldots,\alpha_{n-1}\}$ where  $\alpha_1<\ldots<\alpha_{n-1}=\frob(H)$, and the following holds:\\ There exists $s\in \{1,n\}$ such that for all $j$ with $2\leq j\leq n$,
        \begin{equation*}
        a_n+a_1+\delta=
            \left\{ \,
                \begin{array}{cc}
                a_{j}+a_{n+1-j} & (j<s); \\
                a_{j}+a_{n+2-j} & (j>s);
                \end{array}
            \right.
        \end{equation*}
        where
        \begin{equation*}
        \delta =
            \left\{ \,
                \begin{array}{cc}
                a_2-a_1 & (s=1); \\
                0 & (s\not=1).
                \end{array}
            \right.
        \end{equation*}
    \end{enumerate}
    In particular, $H$ is almost symmetric if and only if $R_H$ is Teter and $\delta_H=0$.
\end{corollary}
\begin{proof}
    We show that $s$ takes values only in $\{1, n\}$.
    Assume $1<s<n$. Then $\delta_H=0$. Letting $j=n$, we get $a_n+a_1=a_n+a_2$, which is a contradiction.
\end{proof}

\begin{corollary}\label{cor:minimalnumefunny}
All numerical semigroup rings with minimal multiplicity and embedding dimension $3$ are Teter.
\end{corollary}
\begin{proof}
    When $H$ has minimal multiplicity, we have $\PF(H)=\{\alpha=a_2-a_1, \frob(H)=a_3-a_1\}$. 
    In this case, considering $t^{\frob(H)+a_2}\omega_{R_H}$, we obtain
    \begin{align*}
        t^{\frob(H)+a_2}\omega_{R_H} 
        &= t^{\frob(H)+a_2}\left(\sum_{h\in H}kt^{h-\alpha}+\sum_{h\in H}kt^{h-\frob(H)}\right) \\
        &= \sum_{h\in H}kt^{h-(a_2-a_1)+(a_3-a_1+a_2)} + \sum_{h\in H}kt^{h+a_2}\\
        &= \sum_{h\in H}kt^{h+a_3} + \sum_{h\in H}kt^{h+a_2} \subsetneq R_H.
    \end{align*}
    Since it is clear that $t^{a_2}, t^{a_3} \in t^{\frob(H)+a_2}\omega_{R_H}$, we conclude that $\frob(H)+a_2 \in T(H)$.
\end{proof}

\begin{corollary}
    Let $H$ be a pseudo-symmetric numerical semigroup of embedding dimension 3.
    Then, $R_H$ is Teter if and only if there exists a pair of generators whose difference is $\frac{\frob(H)}{2}$.
\end{corollary}
\begin{proof}
Assume that $R_H$ is Teter. 
By \autoref{thm:numeOZAKI}, the condition for $s=1,2,3$ is equivalent to
\[
a_3-a_2=\frac{\frob(H)}{2},\quad
a_3-a_1=\frac{\frob(H)}{2},\quad
a_2-a_1=\frac{\frob(H)}{2},
\]
respectively. Hence the assertion follows.
The converse is obtained by reversing the above argument.
\end{proof}

\begin{example}
    Let $H$ be a numerical semigroup with minimal multiplicity. The following conditions are equivalent:
    \begin{enumerate}[\rm (1)]
        \item $\frob(H)+a_1\in T(H)$;
        \item $\delta_H=0$;
        \item $H$ is almost symmetric;
        \item $t^{\gamma_H}\omega_{R_H}=(t^{a_1},t^{a_2},\ldots,t^{a_{n-1}})$.
    \end{enumerate}
\end{example}

Moreover, we provide examples of numerical semigroup rings having arbitrary codimension as follows.

\begin{example}
    For an integer $n \ge 2$ and positive integers $a, s, d > 0$ such that $(a,d)=1$, define a numerical semigroup $H$ as
    $H=\langle a,sa+d, \ldots, sa+nd \rangle$.
    Here, $\emb(H)=n+1$.
    Let $r$ be the remainder and $q$ be the quotient when $a$ is divided by $n$.
    According to \cite[Theorem 3.1 (2)]{numata2014numerical} or \cite{matthews2004numerical}, if $r=1$, then $\type(H)=n$, and in this case $R_H$ is Teter (though not necessarily almost Gorenstein).
\end{example}
\begin{proof}
    In our setting, we have $\PF(H)=\{saq+id-a : (q-1)n+1\leq i \leq qn\}$ (see \cite[Theorem 3.1 (2)]{numata2014numerical} or \cite{matthews2004numerical}).
    By \autoref{thm:numeOZAKI}, if we set $s=1$, the equality
    \[
    \alpha_n+a_2=saq+qnd-a+sa+d=sa+(j-1)d+saq+((q-1)n+n+1-j)d-a=a_j+\alpha_{n+1-j}
    \]
    holds for any $2\leq j\leq n$.
    Therefore, we conclude that $R_H$ is Teter.
\end{proof}

\begin{proposition}\label{propprop:numenume}
    Let $H=\langle n_1,n_2,n_3\rangle$ be a non-symmetric numerical semigroup.
    The following hold:
    \begin{enumerate}[\rm (1)]
        \item If $R_H$ is nearly Gorenstein but not almost Gorenstein, then $R_H$ is Teter;
        \item When $R_H$ is almost Gorenstein, the Hilbert-Burch matrix of $R_H$ can be expressed as
        $A=
        \begin{pmatrix}
            x_1^a & x_2^b & x_3^c\\
            x_2 & x_3 & x_1
        \end{pmatrix}$
        where a,b,c is positive integers.
        In this case, $R_H$ is Teter if and only if $\min\{a,b,c\}=1$.
    \end{enumerate}
\end{proposition}
\begin{proof}
    We show (2) first.
    Since it is a numerical semigroup generated by 3 elements, $H$ is pseudo-symmetric.
    Therefore, by Cor 4, it is sufficient to show the existence of a pair of generators whose difference is $\frac{\frob(H)}{2}$.
    By
    \cite[Theorem 3.3~(3)]{herzog2021canonical}, $H$ can be represented using $a,b,c>0$ as $H=\langle bc+b+1, ac+c+1, ab+a+1 \rangle$ where $(bc+b+1, ac+c+1)=1$.
    Furthermore, $\PF(H)=\{abc-1,2(abc-1)\}$.
    By symmetry, we assume $a\leq b\leq c$.
    (That is, $ab+a+1 \leq ac+c+1\leq bc+b+1$).
    If $a=1$, $R_H$ is clearly Teter, so we show the converse.
Assume that $R_H$ is Teter.
Then $abc-1$ equals one of $bc-ac+b-c$, $bc-ab+b-a$, and $ac-ab+c-a$.
    Assuming equality for (1) or (3) leads to a contradiction with the fact that the embedding dimension of $H$ is 3.
    Therefore, (2) implies $a=1$.
\end{proof}
\begin{remark}
    \begin{enumerate}[\rm(1)]
        \item 
        \autoref{propprop:numenume} generally does not hold when $\emb(H)\geq 4$.
        In fact, we can check that $\langle 11,12,14,15\rangle$ is a non-almost Gorenstein nearly Gorenstein ring satisfying $\type(H)=\emb(H)-1$, but it is not Teter.
        Actually, this numerical semigroup has $\PF(H)=\{13, 31, 32\}$ as its pseudo-Frobenius numbers, but by \autoref{thm:numeOZAKI}, it can not be Teter.
        \item 
        Also, even if the semigroup has minimal multiplicity when the embedding dimension is 4 or more, it is not necessarily Teter.
        Indeed, $H = \langle 4, 7, 9, 10 \rangle$ has minimal multiplicity and has $\PF(H)=\{3, 5, 6\}$ as its pseudo-Frobenius numbers. 
        However, similarly to the above, it follows from \autoref{thm:numeOZAKI} that it can not be Teter.
    \end{enumerate}
\end{remark}
\begin{proposition}\label{prop:cor:minimalnumefunny1}
    If $H$ has minimal multiplicity and $[\tr(\omega_{R_H})]_{a_1}\not= (0)$, then $R_H$ is Teter.
\end{proposition}
\begin{proof}
Assume that $H$ has minimal multiplicity and that $[\tr(\omega_{R_H})]_{a_1}\neq 0$.
Then $a_1\in \tr(H)$, so there exist $i\in [n-1]$, $h\in H$ and $z\in\Omega_H^{-}$ such that
$a_1 = h - \alpha_i + z$,
and in particular $a_1+\alpha_i \in \Omega_H^{-}$ for some $i\in [n-1]$.
By the description of $\Omega_H^{-}$, this means that for some $i\in [n-1]$ we have
$a_1+\alpha_i-\alpha_j \in H \quad\text{for all } j\in [n-1]\setminus\{i\}$.
Using the characterization of such an index, we may take $i=n-1$, and hence
$a_1+\alpha_{n-1}-\alpha_j \in H \quad\text{for all } j\in [n-1]$.
All the elements $a_1+\alpha_{n-1}-\alpha_j$ lie in the Ap\'ery set $\Ap(H,a_1)$,
where $j\in [n-1]$.
By the known structure of $\Ap(H,a_1)$ in this case, we obtain
$a_1+\alpha_{n-1}-\alpha_j = a_1+\alpha_{n-1-j}$
for every  $j\in [n-2]$.
Cancelling $a_1$ gives
$\alpha_{n-1} = \alpha_{n-1-j}+\alpha_j$ for all $j\in [n-2]$.
By \autoref{numerical_cor1}, this condition implies that $H$ is Teter (and hence almost symmetric), as required.
\end{proof}

\begin{remark}
From the proof above, if $H$ has minimal multiplicity,
then $[\tr(\omega_{R_H})]_{a_1}\neq 0$ holds if and only if $R_H$ is almost Gorenstein.
\end{remark}

\section*{Acknowledgments}
The first author was supported by JSPS KAKENHI Grant Number 24K16909.
The authors would like to thank Yuya Otake for pointing out typos.



\end{document}